\documentclass[oneside,english]{amsart}
\usepackage[T1]{fontenc}
\usepackage[latin9]{inputenc}
\usepackage{mathrsfs}
\usepackage{amstext}
\usepackage{amsthm}
\usepackage{amssymb}
\usepackage[all]{xy}

\makeatletter
\numberwithin{equation}{section}
\numberwithin{figure}{section}

\usepackage[dvipsnames,svgnames,x11names,hyperref]{xcolor}
\usepackage{lmodern}
\renewcommand*{\epsilon}{\varepsilon}
\usepackage{amssymb}
\usepackage{amsfonts}
\usepackage{egothic}
\usepackage{xcolor}
\usepackage[T1]{fontenc}
\usepackage{url}
\usepackage{amsthm}
\usepackage{aliascnt}
\usepackage{lipsum}
\usepackage{mathrsfs}
\usepackage{esint}
\usepackage{url}
\usepackage{amsmath}
\usepackage{dsfont}
\usepackage{bbm}
\usepackage{pdflscape}
\usepackage{bbm}
\usepackage{bussproofs}

\DeclareMathOperator{\supp}{supp}
\DeclareMathOperator{\dom}{dom}
\DeclareMathOperator{\Perm}{Perm}

\makeatletter
\def\matrixobject@{%
  \edef \next@{={\DirectionfromtheDirection@ }}%
  \expandafter \toks@ \next@ \plainxy@
  \let\xy@@ix@=\xyq@@toksix@
  \xyFN@ \OBJECT@}
\let\xy@entry@@norm=\entry@@norm
\def\entry@@norm@patched{%
  \let\object@=\matrixobject@
  \xy@entry@@norm }
\AtBeginDocument{\let\entry@@norm\entry@@norm@patched}
\makeatother
\newcommand{\twocong}[2][0.5]{\ar@{}[#2] \save ?(#1)*{\cong}\restore}
\newcommand{\twoeq}[2][0.5]{\ar@{}[#2] \save ?(#1)*{=}\restore}
\newcommand{\ltwocell}[3][0.5]{\ar@{}[#2] \ar@{=>}?(#1)+/r 0.15cm/;?(#1)+/l 0.15cm/^{#3}}
\newcommand{\rtwocell}[3][0.5]{\ar@{}[#2] \ar@{=>}?(#1)+/l 0.15cm/;?(#1)+/r 0.15cm/^{#3}}
\newcommand{\utwocell}[3][0.5]{\ar@{}[#2] \ar@{=>}?(#1)+/d  0.15cm/;?(#1)+/u 0.15cm/_{#3}}
\newcommand{\dtwocell}[3][0.5]{ \ar@{}[#2] { \ar@{=>}?(#1)+/u  0.15cm/;?(#1)+/d 0.15cm/^{#3}}}
\newcommand{\ultwocell}[3][0.5]{\ar@{}[#2] \ar@{=>}?(#1)+/dr  0.15cm/;?(#1)+/ul 0.15cm/^{#3}}
\newcommand{\urtwocell}[3][0.5]{\ar@{}[#2] \ar@{=>}?(#1)+/dl  0.15cm/;?(#1)+/ur 0.15cm/^{#3}}
\newcommand{\dltwocell}[3][0.5]{\ar@{}[#2] \ar@{=>}?(#1)+/ur  0.15cm/;?(#1)+/dl 0.15cm/^{#3}}
\newcommand{\drtwocell}[3][0.5]{\ar@{}[#2] \ar@{=>}?(#1)+/ul  0.15cm/;?(#1)+/dr 0.15cm/^{#3}}

\newcommand{\myar}[2]{\ar^-{#1}[#2]}
\newcommand{\myard}[2]{\ar_-{#1}[#2]}
\usepackage[colorlinks=true,linkcolor=Dark Red, citecolor=Dark Green,linktoc=all]{hyperref}
\usepackage{amsmath}
\usepackage{cleveref}
\makeatletter
  \def\make@df@tag@@#1{%
    \gdef\df@tag{%
      \maketag@@@{\Hy@make@anchor#1}%
      \def\@currentlabel{#1}%
      \def\cref@currentlabel{[equation][2147483647][]#1}%
    }%
  }
  \def\make@df@tag@@@#1{%
    \gdef\df@tag{%
      \tagform@{\Hy@make@anchor#1}%
      \toks@\@xp{\p@equation{#1}}%
      \edef\@currentlabel{\the\toks@}%
      \edef\cref@currentlabel{[equation][2147483647][]\the\toks@}
    }%
  }
\makeatother
\AtBeginDocument{\numberwithin{thm}{subsection}}

\makeatother

\theoremstyle{plain}
\newtheorem{thm}{\protect\theoremname}
\theoremstyle{definition}
\newtheorem{defn}[thm]{\protect\definitionname}
\theoremstyle{remark}
\newtheorem{rem}[thm]{\protect\remarkname}
\theoremstyle{plain}
\newtheorem{prop}[thm]{\protect\propositionname}
\newtheorem{lem}[thm]{\protect\lemmaname}
\newtheorem{cor}[thm]{\protect\corollaryname}
\theoremstyle{definition}
\newtheorem{example}[thm]{\protect\examplename}
\usepackage{babel}
\providecommand{\corollaryname}{Corollary}
\providecommand{\definitionname}{Definition}
\providecommand{\examplename}{Example}
\providecommand{\lemmaname}{Lemma}
\providecommand{\propositionname}{Proposition}
\providecommand{\remarkname}{Remark}
\providecommand{\theoremname}{Theorem}

\begin{document}
\date{\today}
\title{Locally subcartesian closed categories}
\author{Charles Walker}
\email{\tt{charles.walker.math@gmail.com}}
\begin{abstract}
We introduce locally subcartesian closed categories: categories with
pullbacks equipped with a coherent choice of subobjects of pullbacks,
such that the resulting affine base-change functors have right adjoints.
We develop the basic theory, emphasizing the analogy with locally
cartesian closed categories, and justify the terminology by showing
that every slice category is monoidal closed with jointly monic projections.
We also extend the theory of polynomials, showing that such a category
gives rise to a bicategory of polynomials in the sense of Street,
biequivalent to a 2-category of subcartesian polynomial functors.
We illustrate this theory with the Lawvere quantale as a basic example
and the category of nominal sets as a richer one. This work suggests
a natural categorical semantics for an extensional dependent affine
type theory with bunched contexts and affine implication.
\end{abstract}

\maketitle
\tableofcontents{}

\section{Introduction}

Locally cartesian closed categories (lcccs), which are those categories
$\mathcal{E}$ for which each slice category $\mathcal{E}/X$ is cartesian
closed, play a fundamental role in category theory. Most importantly,
they provide a framework for studying polynomial functors \cite{gambinokock},
a categorical semantics for extensional dependent type theory \cite{Seely},
and are an essential component of the logic of topoi. However, requiring
cartesian closure on the slice categories is too restrictive a condition
to accommodate more general structures. One basic example of such
an $\mathcal{E}$ (also used as a base of enrichment for extended
metric spaces) is Lawvere's quantale \cite{lavmetric} of non-negative
reals $\left(\left[0,\infty\right],\geq\right)$ where each slice
category $\mathcal{E}/X$ has a closed monoidal structure given by
the tensor\footnote{We define this tensor for finite values, though it clearly may be
extended to infinite values.}
\[
\xymatrix@=1em{A\ar[rd] &  & B\ar[ld] & \ar@{}[rd]|-{\overset{\otimes_{X}}{\longmapsto}} &  & A+B-X\ar[d]\\
 & X &  &  &  & X
}
\]
Moreover, the monoidal closure of each slice category $\mathcal{E}/X$
is exhibited by the fact $A+B-X\geq C$ if and only if $A\geq\max\left(C-B+X,X\right)$\footnote{Note each non-trivial $X$ is neither a monoid nor comonoid object
here.}. This monoidal structure is distinct from the cartesian structure
given by $\max\left(A,B\right)$, closed on each slice over $X$ via
the Heyting implication $\left[B,C\right]_{X}=X$ for $B\geq C$ and
$C$ otherwise.

In order to give a broader framework which can accommodate such structures
we will need to find a suitable generalization of lcccs. The first
thought is likely to consider generalizations of lcccs in which we
have closed monoidal structures on the slice categories. However,
as with lcccs, it is important these slice-wise monoidal structures
are determined by a suitable base change, similar to the cartesian
case where we have $g\times_{Y}f\cong\Sigma_{f}\Delta_{f}\left(g\right)$
where $\Sigma_{f}$ is composition with $f$ and $\Delta_{f}$ is
pullback along $f$. This special property of having a correspondence
between the monoidal structure on the slice categories and a coherent
base change is why lcccs have two descriptions, namely as those categories
with closed cartesian monoidal structures on the slice categories
or equivalently as those where each base change $\Delta_{f}$ has
a right adjoint $\Pi_{f}$.

More generally, it turns out in order for a family of monoidal slice
categories $\left(\mathcal{E}/X,\otimes_{X},I_{X}\right)$ to give
such a base change (with components denoted $\nabla_{f}$) the monoidal
structures on each slice category must be semicartesian. This means
each monoidal unit $I_{X}$ must be isomorphic to the terminal object
of the slice category, namely the identity morphism. Equivalently,
this property means there are coherent projection maps from each tensor
product.

In addition to this, in order for the category to have a well behaved
logic, so that each affine base change $\nabla_{f}$ restricts to
functors $\mathbf{Sub}\left(Y\right)\to\mathbf{Sub}\left(X\right)$,
we will require that each tensor product on a slice $\mathcal{E}/X$
satisfies uniqueness (but not necessarily existence) of mediating
morphisms as in
\[
\xymatrix@=1em{ &  & T\ar[rrd]^{\gamma_{2}}\ar[lld]_{\gamma_{1}}\ar@{..>}[d]^{!}\\
A &  & A\otimes_{X}B\myard{\pi_{2}}{rr}\myar{\pi_{1}}{ll} &  & B
}
\]
with these ``subproducts'' giving a basic example of coherent sublimits
(which are defined as subterminal cones)\footnote{Like weak limits, sublimits are not unique. This means we must specify
a choice in order to use them, and this choice should be coherent.}. These subproducts in the slice categories $\mathcal{E}/X$ are subpullbacks
in $\mathcal{E}$, which are chosen commuting squares as below

\[
\xymatrix@=1em{T\ar@/^{1pc}/[rdrr]^{\gamma_{2}}\ar@/_{1pc}/[rdd]_{\gamma_{1}}\ar@{..>}[rd]^{!}\\
 & A\otimes_{X}B\ar[rr]\ar[d]_{\pi_{1}}\myar{\pi_{2}}{rr} &  & B\ar[d]^{g}\\
 & A\ar[rr]_{f} &  & X
}
\]
such that given any commuting square as on the outside, mediating
morphisms are unique whenever they exist. In the terminology of Lawvere
\cite{lawvere1973perugia}, such squares are called subcartesian.

A \emph{locally subcartesian category} is a choice of subpullbacks
where identities and binary composition are respected both horizontally
and vertically, consequently giving rise to an indexed monoidal structure
$\nabla_{\left(-\right)}\colon\mathcal{E}^{\textnormal{op}}\to\mathbf{MonCat}$.
The existence of all necessary morphisms is recovered from this coherence,
rather than a universal property. 

The variant without uniqueness, \emph{locally semicartesian categories},
are shown to be canonical in that they are equivalent to bicategory
structures on hom-categories of spans, and thus provide the slice-wise
analogue of monoidal categories generalizing cartesian monoidal categories.

The main subject of this paper, a \emph{locally subcartesian closed
category}, is then such a category with pullbacks, where (with $\Sigma_{f}$
being composition with $f$) the subcartesian monoidal structure $g\otimes_{Y}f\cong\Sigma_{f}\nabla_{f}\left(g\right)$
on each slice category is right closed. The latter condition is equivalent
to asking each affine base change $\nabla_{f}$ has a right adjoint
$\boxtimes_{f}$ called the dependent subproduct along $f$. This
structure is concisely summarized by replacing the usual adjoint triple
of lcccs by the situation where one has a middle comparison in the
adjoint bridge $\Sigma_{f}\dashv\Delta_{f}\Leftarrow\nabla_{f}\dashv\boxtimes_{f}$
given by the inclusions of subpullbacks into pullbacks.

Importantly in such categories, one still has a polynomial calculus
similar to that of lcccs, and we may derive what this calculus should
be from Street's theory of protocalibrations \cite{proto}. Here the
diagrams of the form
\[
\xymatrix@=1em{I &  & E\ar[rr]^{p}\ar[ll]_{s} &  & B\ar[rr]^{t} &  & J}
\]
assemble into a bicategory of polynomials much like the usual setting,
but with slightly different composition and 2-cells. For instance
the polynomial 2-cells now involve subpullbacks rather than pullbacks,
and composition of polynomials now involves pullbacks, subpullbacks,
and ``distributivity subpullbacks'' (the appropriate analogue of
Weber's distributivity pullbacks \cite{weber}). We establish analogues
of the standard Beck and distributivity isomorphisms, in order to
show that subcartesian polynomials and 2-cells are in biequivalence
with subcartesian polynomial functors and cartesian strong natural
transformations defined by sending such a polynomial to the composite
\[
\xymatrix@=1em{\mathcal{E}/I\ar[rr]^{\Delta_{s}} &  & \mathcal{E}/E\ar[rr]^{\boxtimes_{p}} &  & \mathcal{E}/B\ar[rr]^{\Sigma_{t}} &  & \mathcal{E}/J}
\]
and that this biequivalence restricts to subcartesian 2-cells and
cartesian strong cartesian natural transformations. Because subcartesian
polynomial functors preserve monomorphisms, these cartesian strong
transformations also respect a canonical subcartesian strength, and
so may be called \emph{bunched strong}.

Finally, since the Lawvere quantale example is somewhat basic, we
give a more substantial example: the category of nominal sets, equivalent
to the Schanuel topos. We show that it carries a locally subcartesian
closed structure, and in particular a canonical dependent subproduct,
thereby exhibiting a seemingly unnoticed aspect of the resource-sensitive
structure on nominal sets. This yields non-trivial examples of subcartesian
polynomial functors and suggests possible applications in computer
science.

There is of course much more to generalize than that covered in this
paper, as the theory of locally cartesian closed categories and polynomial
functors is quite large \cite{kocknotes,SpivakBook}. We therefore
focus on laying the foundations for further developments. Many readers
will be interested in the type theoretic viewpoint, that these structures
should provide a generalization of Seely's equivalence \cite{Seely}
between dependent type theory and locally cartesian closed categories.
This would be based on the well known fact that semicartesian closed
categories provide a semantics for affine type systems (substructural
systems where one has weakening but not contraction, equivalent to
requiring each resource is used at most once), with the caveat that
we have both cartesian and affine substitution. In the setting where
one also has dependent products (which is the case in both examples),
this gives a 1-categorical affine variant of Riley\textquoteright s
dependent bunched type theory \cite{riley}. The main advantage of
our semantics is that we retain a rich theory of polynomial functors
and that it is provably canonical, providing resource-aware foundations
that remain similar to Martin-Löf type theory \cite{Hofmann,martinlof}.
In this setting a proof no longer gives a truth that may be used arbitrarily
many times, but rather a resource to be used at most once.

\subsection{Structure of the paper}

In Section \ref{sublimprod} we consider the elementary notion of
a sublimit, and specialize it to subproducts. The point is not merely
to weaken products, but to require coherent choices: a choice of subproducts
should define a monoidal structure, just as a choice of subpullbacks
will later define an indexed monoidal structure.

In Section \ref{locsubcart} we define locally subcartesian categories
in terms of coherent chosen subpullbacks, from which the affine base-change
pseudofunctor is constructed. We then show that each slice inherits
a subcartesian monoidal structure, that reindexing is strong monoidal,
and explain the correspondence with bicategory structures on spans.

In Section \ref{locsubcartclosed} we pass to the closed version of
the theory, where the ambient category has pullbacks and the subcartesian
monoidal structure on each slice category is right closed. This right
closure is shown to be equivalent to each affine base-change having
a right adjoint, the dependent subproduct. To study this dependent
subproduct, we replace Weber's distributivity pullbacks \cite{weber}
with distributivity subpullbacks. We establish the corresponding Beck
and distributivity conditions for this data.

In Section \ref{bicatyspanpoly} we turn to polynomials. Using the
bicategory of spans constructed from coherent subpullbacks (and the
interpretation of polynomials as spans \cite{proto,polyspans}) we
define the bicategory of polynomials associated to a locally subcartesian
closed category and describe its composition. We then define the corresponding
subcartesian polynomial functors, equipping them with a ``bunched
strength'' consisting of both a cartesian and subcartesian strength.
Finally, we show in Theorem \ref{mainbiequiv} that the bicategory
of polynomials is biequivalent to the 2-category of subcartesian polynomial
functors and natural transformations which are bunched strong (meaning
they respect both cartesian and subcartesian strength).

In Section \ref{nom} we apply the theory to nominal sets. We define
separated pullbacks relative to support in the base and show that
they form coherent subpullbacks. We then construct the dependent separated
product and verify that it is right adjoint to affine base change,
so that nominal sets are locally subcartesian closed. Finally, we
show that Pitts\textquoteright{} nominal signature for the untyped
lambda calculus is naturally expressed as a polynomial functor in
the subcartesian sense.

In Section \ref{future} we discuss some possible future directions
of research, ranging from semantics of dependent affine type theory
and recognition theorems for subcartesian polynomial functors to bases
of enrichment and lopoi. Moreover, we will study fibred convolution
structures on categories of presheaves over a locally subcartesian
category (as well as in more general settings) in subsequent work.
This should broaden the applications of the theory, particularly to
sheaves for atomic topologies.

\subsection{Acknowledgements}

The author thanks the members of the Masaryk University, TalTech,
and Macquarie University seminars for their questions and comments.
The author also thanks Nathanael Arkor for correcting some typos,
Andy Pitts for drawing attention to constructive versions of nominal
sets, and Richard Garner for pointing out additional examples.

\section{Sublimits and subproducts\label{sublimprod}}

In this section we study structures in which the usual cartesian product
or pullback (product in the slice category) is replaced with a coherent
choice of subobjects of products. Here coherent means this subproduct
should yield a monoidal category structure and a subpullback should
give an indexed monoidal category structure. Note that we must also
make sure our definition works in the examples where the product does
not exist, whilst remaining equivalent to such structures when the
product does exist.

We also remind the reader that sublimits like weak limits are not
unique in general, thus one must specify a choice in order to reason
about them, and any decent behavior will require this choice to be
coherent. Surprisingly, the use of weak limits is relatively common
whereas that of sublimits is not (with the exception of subterminal
objects). This is likely due to a common misconception that existence
is required all of the time to find a non-trivial structure, when
in fact you often only need existence in certain well-behaved situations
where the existence property may be instead established manually. 
\begin{defn}
A \emph{sublimit}\footnote{In the terminology of Freyd sublimits are instead called partial limits.}
over a diagram is a cone over that diagram such that any mediating
morphism from another cone is unique if it exists.
\end{defn}

\begin{rem}
This is equivalent to a subterminal object in the category of cones
over the given diagram.
\end{rem}

In addition to the above remark, the following proposition also justifies
the \emph{sub} prefix in our naming convention.
\begin{prop}
Suppose the limit $\left(V,v_{i}\right)_{i}$ over a diagram exists.
Then sublimits $\left(S,s_{i}\right)_{i}$ over that diagram are in
bijection with monomorphisms $m\colon S\to V$.
\end{prop}

\begin{proof}
Suppose that we are given a sublimit $\left(S,s_{i}\right)_{i}$ over
the diagram. Then there exists a unique mediating morphism $m\colon S\to V$
coherent with the projections. Now suppose we are given $f,g\colon T\to S$
such that $mf=mg$. Since $s_{i}=v_{i}m$ it follows that $s_{i}f=s_{i}g$,
and so $f$ and $g$ are both unique mediating morphisms into the
sublimit and $f=g$.

Conversely suppose we are given the limit $\left(V,v_{i}\right)_{i}$
and a monomorphism $m\colon S\to V$. Then with $s_{i}:=v_{i}m$ and
a cone over $\left(S,s_{i}\right)_{i}$ with two mediating morphisms
$f,g\colon T\to S$ it follows both $mf$ and $mg$ are mediating
morphisms as a cone over $\left(V,v_{i}\right)_{i}$ so $mf=mg$ and
$f=g$. This is a bijection since as seen in the proof the projection
maps $s_{i}$ are uniquely determined by $m$.
\end{proof}

\subsection{Subcartesian monoidal categories\label{subcartesian} }

The following structure generalizes cartesian monoidal categories
in which the tensor $A\otimes B$ is given by the categorical product
$A\times B$ to the case where $A\otimes B$ is a chosen subproduct
of $A$ and $B$, meaning mediating morphisms are unique but may not
exist in general. Note the choice of subproduct will need to be coherent
in order to arrive at a monoidal category structure. It is often the
case that when products exist we have a symmetry $\gamma$ rendering
commutative 
\[
\xymatrix@=1em{A\otimes B\ar[d]_{m^{A,B}}\ar[rr]^{\gamma^{A,B}} &  & B\otimes A\ar[d]^{m^{B,A}}\\
A\times B\ar[rr]_{\sigma^{A,B}} &  & B\times A
}
\]
and so throughout this paper we mostly consider the symmetric setting.

Our definition has an advantage over just using semicartesian monoidal
categories with the subproduct property, as we deduce naturality of
the unitor, associator and symmetry maps along with the triangle and
pentagon equations as a consequence. The hope is this reduced definition
will be simpler to verify in practice, though we will still show the
definition is equivalent.
\begin{defn}
A \emph{symmetric subcartesian monoidal category} $\left(\mathcal{E},\otimes,I\right)$
consists of:
\begin{itemize}
\item a bifunctor $\otimes\colon\mathcal{E}\times\mathcal{E}\to\mathcal{E}$;
\item projections $\pi^{A,B}_{1}\colon A\otimes B\to A$ and $\pi^{A,B}_{2}\colon A\otimes B\to B$
for all $A,B\in\mathcal{E}$;
\item for all $A\in\mathcal{E}$ isomorphisms $\theta^{A}\colon A\cong A\otimes I$
and $\phi^{A}\colon A\cong I\otimes A$;
\item for all $A,B,C\in\mathcal{E}$ isomorphisms $\alpha^{A,B,C}\colon\left(A\otimes B\right)\otimes C\cong A\otimes\left(B\otimes C\right)$\footnote{Here the direction of the associator is chosen based on the direction
of the unitor to match the skew setting.};
\item for all $A,B\in\mathcal{E}$ isomorphisms $\gamma^{A,B}\colon A\otimes B\cong B\otimes A$;
\end{itemize}
such that:
\begin{enumerate}
\item (tensors are subproducts) any filler as below is unique
\[
\xymatrix@=1em{ &  & T\ar@{..>}[d]^{!}\ar[rrd]^{t_{2}}\ar[lld]_{t_{1}}\\
A &  & A\otimes B\myard{\pi^{A,B}_{2}}{rr}\myar{\pi^{A,B}_{1}}{ll} &  & B
}
\]
\item (projections are natural) all morphisms $f\colon A\to A'$ and $g\colon B\to B'$
render commutative\footnote{This is a usual universal property with both existence and uniqueness.
One may define the binary operation from this property.}
\[
\xymatrix@=1em{A\ar[d]_{f} &  & A\otimes B\myar{\pi^{A,B}_{2}}{rr}\ar[d]_{f\otimes g}\myard{\pi^{A,B}_{1}}{ll} &  & B\ar[d]^{g}\\
A' &  & A'\otimes B'\myard{\pi^{A',B'}_{2}}{rr}\myar{\pi^{A',B'}_{1}}{ll} &  & B'
}
\]
\item (units respect projections) all $A$ and $B$ render commutative
\begin{equation}
\xymatrix@=1em{ &  & A\ar[d]^{\theta^{A}}\ar[lld]_{\textnormal{id}_{A}} &  & B\ar[d]_{\phi^{B}}\ar[rrd]^{\textnormal{id}_{B}}\\
A &  & A\otimes I\myar{\pi^{A,I}_{1}}{ll} &  & I\otimes B\myard{\pi^{I,B}_{2}}{rr} &  & B
}
\label{unitproj}
\end{equation}
\item (associators respect projections) all $A,B$ and $C$ render commutative
\begin{equation}
\xymatrix@=1em{ &  &  & \left(A\otimes B\right)\otimes C\myar{\alpha^{A,B,C}}{d}\ar[llld]_{\pi^{A\otimes B,C}_{1}}\ar[rrdr]^{\pi^{A,B}_{2}\otimes C}\\
A\otimes B &  &  & A\otimes\left(B\otimes C\right)\ar[rrr]_{\pi^{A,B\otimes C}_{2}}\ar[lll]^{A\otimes\pi^{B,C}_{1}} &  &  & B\otimes C
}
\label{assocproj}
\end{equation}
\item (symmetries respect projections) all $A$ and $B$ render commutative
\[
\xymatrix@=1em{ &  & A\otimes B\ar[d]^{\gamma^{A,B}}\ar@/_{0.2pc}/[lld]_{\pi^{A,B}_{2}}\ar@/^{0.2pc}/[rrd]^{\pi^{A,B}_{1}}\\
B &  & B\otimes A\myar{\pi^{B,A}_{1}}{ll}\myard{\pi^{B,A}_{2}}{rr} &  & A
}
\]
\end{enumerate}
\end{defn}

\begin{lem}
For any subcartesian monoidal category the unit object is terminal.
\end{lem}

\begin{proof}
For all objects $A$ a morphism $A\to I$ may be constructed as the
composite of the unitor and projection map $A\cong A\otimes I\to I$,
so $I$ is weakly terminal. Moreover to see $I$ is subterminal note
that for any $f\colon A\to I$ we have the commuting diagram
\[
\xymatrix@=1em{I\ar[d]_{\textnormal{id}_{I}} &  & I\otimes A\myar{\pi^{I,A}_{2}}{rr}\ar[d]_{I\otimes f}\myard{\pi^{I,A}_{1}}{ll} &  & A\ar[d]^{f}\\
I &  & I\otimes I\myard{\pi^{I,I}_{2}}{rr}\myar{\pi^{I,I}_{1}}{ll} &  & I
}
\]
and since $\pi^{I,I}_{1}$ is invertible $I\otimes f$ is uniquely
determined by the left square. Moreover, since $\pi^{I,A}_{2}$ is
invertible $f$ is uniquely determined by $I\otimes f$ as in the
right square.
\end{proof}

\begin{prop}
All coherence data of a (symmetric) subcartesian monoidal category
is uniquely determined by the subproduct property.
\end{prop}

\begin{proof}
Since $I$ is the terminal object, we may extend the left diagram
of \eqref{unitproj} by including the projection $\pi^{A,I}_{2}\colon A\otimes I\to I$
and the unique map $!\colon A\to I$, and similarly for the right
diagram. Hence the unit maps are unique fillers. The associator is
seen as the unique filler when \eqref{assocproj} is composed with
the projection $\pi^{A,B}_{1}\colon A\otimes B\to A$. Similarly,
the symmetry is the unique filler with respect to the swapped projections.
\end{proof}

\begin{rem}
The left triangle of \eqref{assocproj} can be replaced by its postcomposite
with the projection $\pi^{A,B}_{1}\colon A\otimes B\to A$, since
the original triangle may then be recovered from the subproduct property.
\end{rem}

\begin{rem}
The projection maps themselves are uniquely determined by those involving
a unit object. Here a general $\pi^{A,B}_{1}$ is given by the composite
\[
\xymatrix@=1em{A\otimes B\myar{A\otimes!_{B}}{rr} &  & A\otimes I\myar{\pi^{A,I}_{1}}{rr} &  & A}
\]
and similarly for the right projections. Moreover, such unit projections
are inverse to the unitors.
\end{rem}

\begin{cor}
The data of a subcartesian monoidal category is coherent with respect
to naturality and the axioms of a symmetric monoidal category.
\end{cor}

\begin{proof}
As we have already shown all data is defined as the unique filler
with respect to projection maps the remaining proof that this gives
rise to a symmetric monoidal category is the same as the proof for
the usual cartesian case (but with the notation $A\times B$ replaced
by $A\otimes B$).
\end{proof}

\begin{cor}
Subcartesian monoidal categories are in bijection with semicartesian
symmetric monoidal categories with jointly monic projections.
\end{cor}

\begin{proof}
We have already shown how these give rise to symmetric monoidal categories
with a terminal unit and jointly monic projections (the subproduct
property). Also recall here the correspondence between semicartesian
categories and monoidal categories with projections, the dual of \cite[Theorem 3.5]{semicart}.
The converse, that all relevant coherence axioms for a subcartesian
monoidal category are satisfied is an instance of the coherence theorem
for symmetric monoidal categories \cite{MacLane}.
\end{proof}

\begin{rem}
It is often useful to describe the subproduct property in terms of
representables as $\mathcal{E}\left(-,X\otimes Y\right)\hookrightarrow\mathcal{E}\left(-,X\right)\times\mathcal{E}\left(-,Y\right)$
especially in the setting of convolution structures \cite{dayconvolution}.
\end{rem}

The following examples are of the dual form, where one has a monoidal
structure on $\mathcal{E}$ with an initial unit and jointly epic
coprojections into the tensor product, thereby giving a subcartesian
monoidal structure on $\mathcal{E}^{\textnormal{op}}$. 
\begin{example}
A cover relation on a category $\mathcal{E}$ satisfying Janelidze's
axioms \cite{CoverRel} gives rise to a subcartesian monoidal structure
on $\mathcal{E}^{\textnormal{op}}$.
\end{example}

\begin{example}
Consider the category $\mathbf{Mnd}_{f}\left(\mathbf{Set}\right)$
of finitary monads on $\mathbf{Set}$. Here the commuting tensor product
of monads defines a symmetric subcartesian monoidal structure on $\mathbf{Mnd}_{f}\left(\mathbf{Set}\right)^{\textnormal{op}}$
with the identity monad as the monoidal unit.
\end{example}

\section{Locally subcartesian categories\label{locsubcart}}

Whilst it may seem best to define locally subcartesian categories
in terms of the name (meaning in terms of the subcartesian monoidal
structures on each slice category, or as an indexed subcartesian monoidal
category) this would not give the simplest definition. This is since
such a definition involves superfluous data as the monoidal structures
may be constructed from the base change. Nevertheless, we will later
give an equivalent formulation closer to this viewpoint. We instead
use the following simple formulation. Moreover, we assume without
loss of generality that units are strictly respected, justified by
the standard equivalence of pseudofunctors to strictly normal ones.
We still use the notation $\otimes$ as in monoidal categories, but
with the understanding that we must now specify the base in addition
to the usual left and right inputs. Throughout this section we establish
existence by coherence\emph{ }rather than by a universal property.
\begin{defn}
\label{locsemidef} A \emph{locally symmetric subcartesian category}
$\left(\mathcal{E},\otimes\right)$ is a category $\mathcal{E}$ equipped
with for each cospan $f\colon A\to X$ and $g\colon B\to X$ in $\mathcal{E}$
a chosen subpullback as below
\[
\xymatrix@=1em{A\otimes_{X}B\ar[d]_{g'}\myar{f'}{rr} &  & B\ar[d]^{g}\\
A\ar[rr]_{f} &  & X
}
\]
symmetric in $f$ and $g$\footnote{Symmetric means we have isomorphisms $A\otimes_{X}B\cong B\otimes_{X}A$
coherent with the projections, and such a symmetry is necessarily
unique here.} and subject to the coherence conditions:
\begin{itemize}
\item (nullary composites) if $g$ is the identity then so is $g'$\footnote{It is standard to ask identities are strictly respected for convenience,
especially in settings involving spans and polynomials.};
\item (binary composites) for two subpullback squares composed as below
\begin{equation}
\xymatrix@=1em{X\otimes_{Y}\left(Y\otimes_{Z}B\right)\ar[d]_{g''}\myar{p'}{rr} &  & Y\otimes_{Z}B\ar[d]_{g'}\myar{q'}{rr} &  & B\ar[d]^{g}\\
X\ar[rr]_{p} &  & Y\ar[rr]_{q} &  & Z
}
\label{binarysubcartesian}
\end{equation}
the resulting outside square is isomorphic to the chosen subpullback,
coherent with the projections;
\end{itemize}
Throughout, we call a commuting square a subpullback if it is (necessarily
uniquely) isomorphic to the chosen subpullback of its cospan.

\end{defn}

\begin{rem}
Note that as we are giving the most general definition we are not
allowed to assume the existence of a terminal object. Indeed, there
are examples of locally cartesian closed categories which are not
cartesian closed, and thus have no terminal object. One simple example
is the poset given by the Lawvere quantale with zero removed, which
is both a lccc and lsccc.
\end{rem}

\begin{rem}
In the non-symmetric setting one needs to add the nullary and binary
axioms for vertical composition. Here it becomes apparent that the
legs of subpullbacks play a different role, namely the horizontal
property gives pseudo-functoriality of base change whereas the vertical
property is needed to define base change on morphisms in slice categories.
\end{rem}

The reader will notice the definition does not require the category
to have pullbacks, and indeed examples such as the following do not
have pullbacks.
\begin{example}
\label{FI} The category of finite sets and functions $\mathbf{FinSet}$
has pushouts (and in this category injections are stable under pushout),
and so $\mathbf{FinSet}^{\textnormal{op}}$ has pullbacks (with opposite-injections
stable under pullback). Thus in the category $\mathbf{FI}^{\textnormal{op}}$
of finite sets and opposite-injections we may define the subpullback
of any cospan to be a chosen pullback square in the ambient category
$\mathbf{FinSet}^{\textnormal{op}}$. Here the ambient uniqueness
of the mediating map is retained through the faithful functor $\mathbf{FI}^{\textnormal{op}}\to\mathbf{FinSet}^{\textnormal{op}}$,
but existence of the mediating map inside $\mathbf{FI}^{\textnormal{op}}$
is not retained, since it need not be an opposite-injection. The coherence
of the pullback squares is also retained.
\end{example}

\begin{example}
Consider the category $\mathbf{FinProb}$ of finite probability spaces,
whose objects are finite positive probability distributions $\left(X,p_{X}\right)$
and morphisms $f\colon\left(A,p_{A}\right)\to\left(X,p_{X}\right)$
are functions on the underlying sets satisfying $p_{X}\left(x\right)=\sum_{a\in f^{-1}\left(x\right)}p_{A}\left(a\right)$.
Given an additional $g\colon\left(B,p_{B}\right)\to\left(X,p_{X}\right)$
the canonical candidate square is given by the pullback of underlying
sets $A\times_{X}B$ with the probability distribution $p\left(a,b\right)=p_{A}\left(a\right)p_{B}\left(b\right)/p_{X}\left(f\left(a\right)\right)$
and the canonical projection maps respect the distribution. Note that
this only gives a subpullback in $\mathbf{FinProb}$ since an underlying
induced mediating function into the pullback of sets does not respect
the probability distribution in general. This is an example of conditional
independence structure as studied by Simpson \cite{SimpsonInd}.
\end{example}

For the interested reader we include the more general notion without
uniqueness, though such categories will be less important due to their
ill behaved logic. Moreover this lack of uniqueness of mediating morphisms
means this notion requires additional coherence axioms.
\begin{defn}
\label{locsemidefnotuni} A \emph{locally semicartesian category }is
a category $\mathcal{E}$ equipped with for each cospan $f\colon A\to X$
and $g\colon B\to X$ in $\mathcal{E}$ a chosen commuting square
(called a chosen semipullback\footnote{This name is justified by such squares corresponding to the semicartesian
structures on slice categories.}) as below
\[
\xymatrix@=1em{A\otimes_{X}B\ar[d]_{g'}\myar{f'}{rr} &  & B\ar[d]^{g}\\
A\ar[rr]_{f} &  & X
}
\]
such that if $f$ is the identity, then so is $f'$ and if $g$ is
the identity then so is $g'$ along with horizontal binary coherence
isomorphisms, which are the identity on both composites
\[
\xymatrix@=1em{P\ar[rr]\ar[d] &  & B\ar[rr]^{\textnormal{id }}\ar[d] &  & B\ar[d] &  &  & P\ar[rr]^{\textnormal{id }}\ar[d] &  & P\ar[rr]^{\textnormal{}}\ar[d] &  & B\ar[d]\\
A\ar[rr] &  & X\ar[rr]_{\textnormal{id}} &  & X &  &  & A\ar[rr]_{\textnormal{id}} &  & A\ar[rr]_{\textnormal{}} &  & X
}
\]
and such that the two ways of constructing the coherence isomorphisms
between a triple of chosen squares and corresponding outside chosen
square
\[
\xymatrix@=1em{A\ar[rr]\ar[d] &  & B\ar[rr]\ar[d] &  & C\ar[rr]\ar[d] &  & D\ar[d] &  &  &  & A\ar[rrrrrr]\ar[d] &  &  &  &  &  & D\ar[d]\\
W\ar[rr] &  & X\ar[rr] &  & Y\ar[rr] &  & Z &  &  &  & W\ar[rr] &  & X\ar[rr] &  & Y\ar[rr] &  & Z
}
\]
give\footnote{One may combine the left two squares or the right two first.}
the same result, with vertical binary coherence isomorphisms satisfying
the analogue of the above. Moreover we require that the two ways of
constructing the coherence isomorphisms between the left $2\times2$
grid and right square\footnote{One may compose horizontally and then vertically, or vertically and
then horizontally.}
\[
\xymatrix@=1em{A\ar[rr]\ar[d] &  & B\ar[rr]\ar[d] &  & C\ar[d] &  &  &  & A\ar[rrrr]\ar[dd] &  &  &  & C\ar[d]\\
P\ar[rr]\ar[d] &  & Q\ar[rr]\ar[d] &  & R\ar[d] &  &  &  &  &  &  &  & R\ar[d]\\
X\ar[rr] &  & Y\ar[rr] &  & Z &  &  &  & X\ar[rr] &  & Y\ar[rr] &  & Z
}
\]
 give the same result.
\end{defn}

\begin{rem}
These categories contain three main types of squares equipped with
extra data. Namely the chosen semipullbacks (with the identity map
on the apex), semipullbacks (which are squares equipped with a coherent
isomorphism to the chosen square), and those squares equipped with
a possibly non-invertible factorization through the chosen square.
\end{rem}

\begin{rem}
The symmetric case is similar to the above, but with additional symmetry
isomorphisms $\gamma^{A,B}_{X}\colon A\otimes_{X}B\cong B\otimes_{X}A$
coherent with the projections. These must horizontally respect naturality,
unitality and associativity. The vertical coherence data is then constructed
from the horizontal coherence data and this symmetry (or vice versa),
and the interchange property becomes redundant.
\end{rem}

\begin{rem}
\label{counitmap} It is worth noting that for any $p\colon E\to B$
and $f\colon A\to B$ we have a counit candidate $\epsilon_{p}\colon\Sigma_{p}\nabla_{p}\left(f\right)\to f$
given by the top morphism of the chosen square, but no diagonal maps
or unit in general.
\end{rem}

\subsection{Coherent subpullbacks}

We will now show that locally subcartesian categories give rise to
an affine base change pseudofunctor $\nabla_{\left(-\right)}\colon\mathcal{E}^{\textnormal{op}}\to\mathbf{Cat}$.
Whilst the following argument may be viewed as giving basic properties
of the cartesian lifts of the corresponding fibration, staying with
the pseudofunctor viewpoint throughout the paper will be clearer.
We start with the following simple but useful property, namely that
we still have an analogue of the so called pullback lemma. This fact
was noticed by Lawvere \cite{lawvere1973perugia}.
\begin{lem}
[Subpullback lemma]\label{subpasting} For a locally subcartesian
category $\mathcal{E}$, and given diagram
\[
\xymatrix@=1em{S\ar[d]_{g''}\myar{p'}{rr} &  & T\ar[d]_{g'}\myar{q'}{rr} &  & B\ar[d]^{g}\\
X\ar[rr]_{p} &  & Y\ar[rr]_{q} &  & Z
}
\]
where the rightmost square is a chosen subpullback and the left square
commutes, the left square is (canonically isomorphic to) a chosen
subpullback if and only if the outside square is.
\end{lem}

\begin{proof}
We need only check that assuming the outside is a chosen subpullback,
the left square must be (as the converse is by definition). Under
this assumption we note that the outside square is coherently isomorphic
to the composite
\[
\xymatrix@=1em{\overline{S}\ar[d]_{\overline{g''}}\myar{\overline{p'}}{rr} &  & T\ar[d]_{g'}\myar{q'}{rr} &  & B\ar[d]^{g}\\
X\ar[rr]_{p} &  & Y\ar[rr]_{q} &  & Z
}
\]
where the left square is the chosen subpullback of $p$ and $g'$
(since both are the chosen subpullback of $g$ and $qp$). This gives
an isomorphism $\alpha\colon S\to\overline{S}$ with $q'\cdot\overline{p'}\cdot\alpha=q'\cdot p'$
and $\overline{g''}\cdot\alpha=g''$ and we may deduce $\overline{p'}\cdot\alpha=p'$
since the right square is a subpullback.
\end{proof}

\begin{rem}
In a similar setting, Simpson \cite{SimpsonInd} took this property
as an assumption (though without assuming uniqueness of mediating
morphisms). In this case without uniqueness (with ``semipullbacks''),
there is no reason for the coherence property on the outside square
to upgrade to one on the left square.
\end{rem}

The following is a slight improvement to the subpullback lemma which
we will use regularly.
\begin{lem}
\label{spb1} For a locally subcartesian category $\mathcal{E}$,
and given diagram
\[
\xymatrix@=1em{A\ar@{..>}[rr]_{p'}\ar[d]_{g''}\ar@/^{1pc}/[rrrr]^{\left(qp\right)'} &  & B\ar[rr]_{q'}\ar[d]_{g'} &  & C\ar[d]^{g}\\
X\ar[rr]_{p} &  & Y\ar[rr]_{q} &  & Z
}
\]
where the outside square and right square are subpullbacks, there
exists a unique arrow $p'$ making the diagram commute, and consequently
the left square a subpullback.
\end{lem}

\begin{proof}
Take $p'$ to be the $\overline{p'}\cdot\alpha$ as in Lemma \ref{subpasting},
which is a unique filler as the right square is a subpullback.
\end{proof}

Another simple but useful property is that isomorphisms of subpullbacks
respect composition.
\begin{lem}
\label{spb2} Given another diagram as in \eqref{binarysubcartesian}
\[
\xymatrix@=1em{M\ar[d]_{g^{**}}\myar{p^{*}}{rr} &  & N\ar[d]_{g^{*}}\myar{q^{*}}{rr} &  & C\ar[d]^{g}\\
X\ar[rr]_{p} &  & Y\ar[rr]_{q} &  & Z
}
\]
we have unique isomorphisms $\alpha$ and $\beta$ giving a commuting
diagram
\begin{equation}
\xymatrix@=1em{A\ar@/_{1pc}/[dd]_{g''}\myar{p'}{rrr}\ar@{..>}[d]^{\beta} &  &  & B\ar@/_{1pc}/[dd]\sb(.7){g'}\myar{q'}{rrrd}\ar@{..>}[d]^{\alpha}\\
M\ar[d]^{g^{**}}\myar{p^{*}}{rrr} &  &  & N\ar[d]^{g^{*}}\myar{q^{*}}{rrr} &  &  & C\ar[d]^{g}\\
X\ar[rrr]_{p} &  &  & Y\ar[rrr]_{q} &  &  & Z
}
\label{bincoh}
\end{equation}
\end{lem}

\begin{proof}
We construct $\alpha$ from the isomorphism of subpullbacks of $q$
and $g$ (unique since the right is a subpullback) and we then find
a unique $\beta$ from applying the vertical version of Lemma \ref{spb1}
to the composite on the left. 
\end{proof}

We now establish that this gives rise to a base change functor which
behaves as one would expect. The argument and results are the same
as in the standard setting of pullbacks, with the exception that we
must establish the existence of all relevant data manually (since
we do not have the existence property). We also make the observation
that pseudofunctoriality of the affine base change comes from horizontal
composition of squares, whilst the action on morphisms of a slice
category comes from vertical composition. It is thus apparent that
the horizontal and vertical morphisms of a subpullback play a different
role in the theory.
\begin{lem}
\label{affps} Taking subpullbacks along a morphism $f\colon X\to Y$
defines a functor $\nabla_{f}\colon\mathcal{E}/Y\to\mathcal{E}/X$.
Moreover, this data defines a pseudofunctor $\nabla_{\left(-\right)}\colon\mathcal{E}^{\textnormal{op}}\to\mathbf{Cat}$.
\end{lem}

\begin{proof}
On each object $h\colon H\to Y$ the map $\nabla_{f}\left(h\right)$
is given by the subpullback of $h$ along $f$. Given a morphism $p\colon h\to h'$
in the slice category meaning $h'p=h$ we may use composition of chosen
subpullbacks to form the (rotated) diagram
\[
\xymatrix@=1em{S\ar[rd]^{\alpha}\ar@/^{0.7pc}/[rrrrrrrd]^{\nabla_{f}\left(h\right)}\ar@/_{1pc}/[rdd]\\
 & \overline{S}\ar[d]^{f''}\myard{\nabla_{f'}\left(p\right)}{rrr} &  &  & T\ar[d]_{f'}\myard{\nabla_{f}\left(h'\right)}{rrr} &  &  & B\ar[d]^{f}\\
 & X\ar[rrr]_{p} &  &  & Y\ar[rrr]_{h'} &  &  & Z
}
\]
where $\alpha$ is the unique filler and is invertible. Hence $\nabla_{f}\left(p\right)$
is given by $\nabla_{f'}\left(p\right)\cdot\alpha\colon\nabla_{f}\left(h\right)\to\nabla_{f}\left(h'\right)$.
Note if $p$ is the identity on $h$ then $\nabla_{f'}\left(p\right)\cdot\alpha$
is the identity by uniqueness of fillers. Similarly that $\nabla_{f}\left(p_{2}p_{1}\right)=\nabla_{f}\left(p_{2}\right)\cdot\nabla_{f}\left(p_{1}\right)$
follows from the subpullback property of the rightmost square in the
relevant calculation.

The binary pseudofunctoriality constraints are constructed from essentially
the same diagram, and the unit constraints may be taken to be strict.
The pseudofunctor coherence axioms follow from uniqueness of mediating
morphisms.
\end{proof}

\begin{rem}
One may apply the Grothendieck construction to this affine base change
pseudofunctor, obtaining a fibration $\textnormal{cod}\colon\left[2,\mathcal{E}\right]_{\textnormal{fspb}}\to\mathcal{E}$
where the total category has the usual objects given by maps of $\mathcal{E}$
and morphisms given by those squares which factor through the chosen
subpullback\footnote{Another option is to define a subfibration to be a functor equipped
with a coherent choice of subcartesian lifts (with uniqueness but
not existence). Then in this setting codomain is a subfibration on
the full arrow category $\left[2,\mathcal{E}\right]$.}. In the semipullback setting, the morphisms of $\left[2,\mathcal{E}\right]_{\textnormal{fsmpb}}$
are squares equipped with a specified factorization through the semipullback.
In analogy with Simpson's work \cite{SimpsonProbSheaf} our subpullbacks
correspond to independent pullbacks and these subpullback-factoring
squares correspond to independent squares.
\end{rem}

Before extending to an indexed monoidal functor, we must consider
the Beck data.
\begin{prop}
[$\Sigma$-$\nabla$ Beck comparison]\label{beck1} Every locally subcartesian
category comes equipped with, for every commuting square which factors
through the canonical subpullback as on the left below:
\[
\xymatrix@=1em{P\ar[rr]^{p'}\ar[d]_{f'} &  & B\ar[d]^{f} &  &  &  &  & \mathcal{E}/P\ar[d]_{\Sigma_{f'}}\dtwocell[0.5]{rdr}{\mathfrak{b}} &  & \mathcal{E}/B\ar[d]^{\Sigma_{f}}\myard{\nabla_{p'}}{ll}\\
A\ar[rr]_{p} &  & X &  &  &  &  & \mathcal{E}/A &  & \mathcal{E}/X\myar{\nabla_{p}}{ll}
}
\]
a canonical natural transformation (called Beck data) as on the right
above, coherent with respect to nullary and binary pastings. This
transformation is invertible if and only if the original square is
a subpullback.
\end{prop}

\begin{proof}
Suppose we are given such a commuting square which factors through
the subpullback denoted $p'_{s}$ and $f'_{s}$ via a mediating map
$\theta$. Given a morphism $h\colon M\to B$ we form the left diagram
below
\[
\xymatrix@=1em{Q\ar[d]_{\nabla_{p'}\left(h\right)}\ar[rr] &  & Q_{s}\ar[rr]^{p''_{s}}\ar[d]_{\nabla_{p'_{s}}\left(h\right)} &  & M\ar[d]^{h} &  &  &  & Q\ar[d]_{\nabla_{p'}\left(h\right)}\ar[rr] &  & Q_{s}\myar{\mathfrak{b}_{h,s}}{rdr}\ar[d]_{\nabla_{p'_{s}}\left(h\right)}\ar@/^{0.7pc}/[rrdrr]^{p''_{s}}\\
P\ar[rr]_{\theta}\ar@/_{0.5pc}/[rrd]_{f'} &  & P_{s}\ar[rr]|-{p'_{s}}\ar[d]_{f'_{s}} &  & B\ar[d]^{f} &  &  &  & P\ar[rr]_{\theta}\ar@/_{0.7pc}/[rrrrd]_{f'} &  & P_{s}\ar@/_{0.4pc}/[drr]^{f'_{s}} &  & A\otimes_{X}M\myard{\hat{p}}{rr}\ar[d]^{\nabla_{p}\left(fh\right)} &  & M\ar[d]^{fh}\\
 &  & A\ar[rr]_{p} &  & X &  &  &  &  &  &  &  & A\ar[rr]_{p} &  & X
}
\]
The Beck data component at $h$ denoted $\mathfrak{b}_{h}$ is then
given by the composite of the isomorphism $\mathfrak{b}_{h,s}$ (induced
by uniqueness of the subpullback of $fh$ and $p$)  and $Q\to Q_{s}$.
Coherence is a consequence of the uniqueness of such mediating morphisms.
Clearly if we start with a subpullback meaning $\theta$ is invertible
then so is $Q\to Q_{s}$. The converse direction is recovered from
taking $h$ to be the identity morphism.
\end{proof}

\begin{rem}
\label{necunique} An important point about Beck data is that it is
not merely just some coherence data that we must give; it is in fact
a universal construction. For any pair of objects $X$ and $Y$ in
a locally subcartesian category $\mathcal{E}$ with a terminal object
the inclusion of normal form composites $\Sigma_{f}\nabla_{g}$ into
arbitrary composites of $\Sigma_{p}$ and $\nabla_{q}$ in $\mathbf{Cat}$
\[
\textnormal{inc}_{X,Y}\colon N\left(\Sigma,\nabla\right)\left(X,Y\right)\to A\left(\Sigma,\nabla\right)\left(X,Y\right)
\]
has a right adjoint $\textnormal{norm}_{X,Y}$ called normalization.
The Beck data gives each counit component $\epsilon_{F}\colon\textnormal{inc}\;\textnormal{norm}\;F\Rightarrow F$
of this adjoint equivalence. In terms of universal arrows, this implies
that for any natural transformation $\phi$ (respecting strength maps
defined later) we have a unique coherent $\theta$ as on the left
below 
\[
\xymatrix@=1em{\Sigma_{f'}\nabla_{p'}\ar@{..>}[dd]_{\textnormal{inc}\;\theta}\ar[rrdd]^{\phi} &  &  &  &  & T\ar@/^{1pc}/[rdrr]^{p'}\ar@/_{1pc}/[rdd]_{f'}\ar@{..>}[rd]^{\theta}\\
 &  &  &  &  &  & A\otimes_{X}B\ar[d]_{f'_{s}}\myar{p'_{s}}{rr} &  & B\ar[d]^{f}\\
\Sigma_{f'_{s}}\nabla_{p'_{s}}\ar[rr]_{\mathfrak{b}} &  & \nabla_{p}\Sigma_{f} &  &  &  & A\ar[rr]_{p} &  & X
}
\]
for any chosen subpullback square (and outside square) as on the right
above.
\end{rem}

\subsection{Indexed monoidal structure}

We can now use this Beck data to show that base change respects the
canonical monoidal structure. Note that we have chosen the ordering
of our monoidal structure $u\otimes_{B}v:=\Sigma_{v}\nabla_{v}\left(u\right)$
such that we will have monoidal right closure in later sections.
\begin{prop}
Each slice category $\mathcal{E}/B$ is a subcartesian monoidal category
with tensor $u\otimes_{B}v:=\Sigma_{v}\nabla_{v}\left(u\right)$.
Moreover, the base change respects this monoidal structure thus giving
an indexed monoidal category $\nabla_{\left(-\right)}\colon\mathcal{E}^{\textnormal{op}}\to\mathbf{MonCat}$.
\end{prop}

\begin{proof}
For each object $B$, we must give a subcartesian structure on $\mathcal{E}/B$.
The strict symmetry data is by assumption. The unit is the terminal
object in the slice category, namely the identity on $B$. All remaining
data is constructed from the coherence for composition of chosen subpullbacks.
By Lemma \ref{affps} we have an indexed category $\nabla_{\left(-\right)}$,
and so it remains to check each $\nabla_{p}\colon\mathcal{E}/B\to\mathcal{E}/E$
respects the binary constraints of the monoidal structure
\[
\begin{aligned}\nabla_{p}\left(f\otimes_{B}g\right) & =\nabla_{p}\Sigma_{g}\nabla_{g}\left(f\right)\\
 & \cong\Sigma_{g'}\nabla_{p'}\nabla_{g}\left(f\right)\\
 & \cong\Sigma_{g'}\nabla_{g'}\nabla_{p}\left(f\right)\\
 & =\left(\nabla_{p}\left(f\right)\right)\otimes_{E}g'\\
 & =\left(\nabla_{p}\left(f\right)\right)\otimes_{E}\left(\nabla_{p}\left(g\right)\right)
\end{aligned}
\]
That units are respected is trivial.
\end{proof}

\begin{rem}
When $\mathcal{E}$ has pullbacks each slice category $\mathcal{E}/X$
has two monoidal structures; namely the cartesian monoidal structure
$\left(\mathcal{E}/X,\times_{X},\mathbf{1}\right)$ and the subcartesian
monoidal structure $\left(\mathcal{E}/X,\otimes_{X},\mathbf{1}\right)$.
Since the latter monoidal structure is semicartesian we have canonical
oplax monoidal functors $\left(\mathcal{E}/X,\otimes_{X},\mathbf{1}\right)\to\left(\mathcal{E}/X,\times_{X},\mathbf{1}\right)$.
\end{rem}

The following shows that the slice categories must always at least
be semicartesian.
\begin{prop}
\label{alwayssemi} Suppose $\left(\mathcal{E}/X,\otimes_{X},I_{X}\right)$
is a monoidal category for all $X$, with natural isomorphisms $g\otimes_{X}f\cong\Sigma_{f}\nabla_{f}\left(g\right)$
for some special indexed category $\nabla_{\left(-\right)}\colon\mathcal{E}^{\textnormal{op}}\to\mathbf{Cat}$.
Then the monoidal structure on each $\mathcal{E}/X$ is semicartesian.
\end{prop}

\begin{proof}
Let the monoidal unit of $\mathcal{E}/X$ be $I_{X}$. Then 
\[
\nabla_{\textnormal{id}_{X}}\left(I_{X}\right)=\Sigma_{\textnormal{id}_{X}}\nabla_{\textnormal{id}_{X}}\left(I_{X}\right)\cong I_{X}\otimes_{X}\textnormal{id}_{X}\cong\textnormal{id}_{X}
\]
since $I_{X}$ is a unit. Since the pseudofunctor $\nabla_{\left(-\right)}$
must respect units $\nabla_{\textnormal{id}_{X}}\left(I_{X}\right)\cong\textnormal{id}_{\mathcal{E}/X}\left(I_{X}\right)=I_{X}$.
Thus the unit $I_{X}$ is isomorphic to the terminal object.
\end{proof}

This allows for an alternative definition of locally subcartesian
categories, involving a special type of indexed monoidal category.
We now define what ``special'' means for indexed categories and
monoidal categories.
\begin{defn}
\label{spec} An indexed category $\nabla_{\left(-\right)}\colon\mathcal{E}^{\textnormal{op}}\to\mathbf{Cat}$
is \emph{special} if each object $X\in\mathcal{E}$ is assigned to
the slice $\mathcal{E}/X$. An indexed monoidal category (with respect
to strong monoidal functors) $\nabla_{\left(-\right)}\colon\mathcal{E}^{\textnormal{op}}\to\mathbf{MonCat}$
is \emph{special} if the underlying indexed category is special and
the tensor on each $\mathcal{E}/X$ is given by $f\otimes_{X}g=\Sigma_{g}\nabla_{g}\left(f\right)$\footnote{Note the reindexing functors are necessarily strong monoidal here.}.
\end{defn}

The following simple characterization now follows.
\begin{cor}
\label{specmon} A locally semicartesian category $\mathcal{E}$ is
equivalent to a special indexed monoidal category $\nabla_{\left(-\right)}\colon\mathcal{E}^{\textnormal{op}}\to\mathbf{MonCat}$.
\end{cor}

\begin{proof}
We need only explain the converse direction. Given such a special
indexed monoidal category, which is semicartesian on slices by Prop.
\ref{alwayssemi} and thus given by commuting squares, the pseudofunctoriality
of $\nabla_{\left(-\right)}$ gives the axioms for horizontal square
composition. That reindexing is strong monoidal allows one to recover
the vertical coherence data, by the reverse of the semicartesian version
of the construction of Beck data of Prop. \ref{beck1} (and this reindexing
is itself functorial). The interaction of reindexing and pseudofunctoriality
gives the interchange property. 
\end{proof}

\begin{rem}
Given a locally subcartesian category with pullbacks, each $\nabla_{f}$
is canonically a subfunctor of $\Delta_{f}$ meaning the natural transformation
has monic components. This family of natural transformations for each
$f$ comprises an icon \cite{icons} exhibiting $\nabla_{\left(-\right)}$
as a subpseudofunctor of $\Delta_{\left(-\right)}\colon\mathcal{E}^{\textnormal{op}}\to\mathbf{Cat}$.

One curious point here is that the act of pullbacks preserving certain
maps is itself a ``subpullback condition'' under pseudo-natural
transformations rather than icons. Hence one may be motivated to consider
subobjects of $\Delta_{\left(-\right)}$ with respect to oplax natural
transformations as a more general setting.
\end{rem}

\subsection{Spans with subpullbacks}

At first the notion of locally subcartesian categories (or more general
semicartesian variants) may seem artificial. However, it turns out
that this definition is canonical and provably so. 

Monoidal categories (equivalent to one object bicategories) came from
relaxing the assumption that the composition (tensor) structure is
given by a universal property, instead allowing any operation which
is suitably coherent. Following the same idea, instead of restricting
to bicategory structures on hom-categories of spans given by a universal
property (pullback) we allow all possible bicategory structures. We
then find that locally semicartesian categories correspond to bicategory
structures on span. The uniqueness or ``sub'' property is then added
to ensure the category has a well behaved logic of subobjects.

To make this precise we restrict to bicategory structures on span
which satisfy the expected decomposition property.
\begin{defn}
A bicategory structure (meaning a coherent family of composition functors)
on a family of hom-categories of spans $\mathbf{Span}\left(\mathcal{E}\right)\left(X,Y\right)$
is \emph{functorially decomposable} if for all spans we have $\left(s,t\right)\cong\left(s,1\right);\left(1,t\right)$
and moreover the embeddings from $\mathcal{E}$ and $\mathcal{E}^{\textnormal{op}}$
are pseudofunctorial. Without loss of generality, we will consider
the case where these properties hold \emph{strictly}. It is straightforward
but tedious to give the non-strict version.
\end{defn}

\begin{rem}
This condition is actually somewhat restrictive, as some examples
such as relations are only oplax decomposable with maps $\left(s,t\right)\Rightarrow\left(s,1\right);\left(1,t\right)$,
in this case yielding Street's bicategory of polynomial relations
\cite{polyspans}. The issue here is that $\left(s,1\right);\left(1,t\right)$
might not be $\left(s,t\right)$ because there is no guarantee the
resulting span is jointly monic, hence requiring a comparison constructed
from an image factorization. Though such cases still have canonical
comparison maps, it is not clear if they can be simply understood
in our framework.
\end{rem}

\begin{prop}
All strictly functorially decomposable bicategory structures on the
family of hom-categories of spans $\mathbf{Span}\left(\mathcal{E}\right)\left(X,Y\right)$
over a category $\mathcal{E}$ have composition defined by placing
atop a commuting square as below
\[
\xymatrix@=1em{ &  & P\ar[ld]_{g'}\ar[rd]^{f'}\\
 & M\ar[ld]_{p}\ar[rd]^{f} &  & N\ar[ld]_{g}\ar[rd]^{q}\\
X &  & Y &  & Z
}
\]
and composing the left and right legs. Moreover these squares respect
identities and binary composition up to coherent isomorphism. 
\end{prop}

\begin{proof}
Suppose we are given such a bicategory structure. We define $\left(f',g'\right)$
as the composite $\left(1,f\right);\left(g,1\right)$ and then whisker
the 2-cell $\epsilon_{f}$ as below
\[
\xymatrix@=1em{ & M\ar[ld]_{f}\ar[rd]^{f}\ar[dd]_{f} &  & N\ar[ld]_{g}\ar[rd]^{1} &  &  &  &  &  & P\ar[ld]_{fg'}\ar[rd]^{f'}\ar[dd]_{f'}\\
Y &  & Y &  & N &  & = &  & Y &  & N\\
 & Y\ar[ul]^{1}\ar[ru]_{1} &  &  &  &  &  &  &  & N\ar[ul]^{g}\ar[ru]_{1}
}
\]
as the resulting 2-cell must be $f'$ we have shown the square commutes.
If $f$ is the identity then so is $f'$ since composition respects
identity 2-cells. Similarly, composition respects binary composition
$f=f_{2}f_{1}$. The same holds for $g$ by a similar argument.
\end{proof}

We now show that coherent subpullbacks (or semipullbacks) are enough
to form a bicategory of spans, with the non-trivial part being horizontal
composition of 2-cells. The idea is to factor the top middle morphisms
using the 2-cell data, so that after constructing all necessary squares
one square will appear twice thus leading to a simplification. Again,
we establish the existence of necessary morphisms via coherence rather
than a universal property.
\begin{prop}
Given a locally subcartesian (or semicartesian) category $\left(\mathcal{E},\otimes\right)$
there is a bicategory of spans over $\mathcal{E}$ denoted $\mathbf{Span}\left(\mathcal{E},\otimes\right)$
which has objects, 1-cells and 2-cells as usual. Composition of 1-cells
is given by subpullback (or semipullback), and horizontal composition
of two general 2-cells $f$ and $g$ is given by first constructing
the squares in a diagram as below
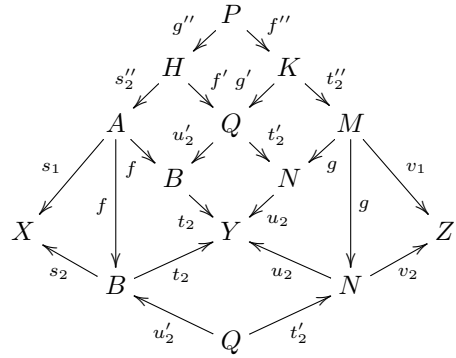
\[
\xymatrix@=0.7em{ &  &  &  & P\ar[rd]^{f''}\ar[ld]_{g''}\\
 &  &  & H\ar[rd]^{f'}\ar[ld]_{s_{2}''} &  & K\ar[ld]_{g'}\ar[rd]^{t_{2}''}\\
 &  & A\ar[ldld]_{s_{1}}\ar[rd]_{f}\ar[ddd]_{f} &  & Q\ar[rd]^{t_{2}'}\ar[ld]_{u_{2}'} &  & M\ar[ld]^{g}\ar[rdrd]^{v_{1}}\ar[ddd]^{g}\\
 &  &  & B\ar[rd]_{t_{2}} &  & N\ar[ld]^{u_{2}}\\
X &  &  &  & Y &  &  &  & Z\\
 &  & B\ar[llu]^{s_{2}}\ar[rru]_{t_{2}} &  &  &  & N\ar[llu]^{u_{2}}\ar[rru]_{v_{2}}\\
 &  &  &  & Q\ar[lul]^{u_{2}'}\ar[urr]_{t_{2}'}
}
\]
and then taking $f'\cdot g''=g'\cdot f''\colon P\Rightarrow Q$ as
the resulting 2-cell\footnote{Here we have suppressed the unique isomorphism between the $2\times2$
grid of squares and its outer square.}.
\end{prop}

\begin{proof}
This horizontal composite of $f\colon\left(s_{1},t_{1}\right)\Rightarrow\left(s_{2},t_{2}\right)$
and $g\colon\left(u_{1},v_{1}\right)\Rightarrow\left(u_{2},v_{2}\right)$
is a 2-cell as it gives a commuting diagram as below
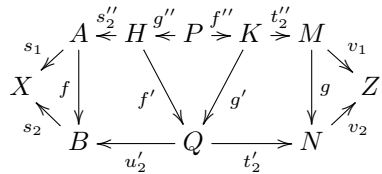
\[
\xymatrix@=0.7em{ & A\ar[ld]_{s_{1}}\ar[dd]_{f} & H\ar[l]_{s_{2}''}\ar[rdd]_{f'} & P\ar[r]^{f''}\ar[l]_{g''} & K\ar[r]^{t_{2}''}\ar[ldd]^{g'} & M\ar[rd]^{v_{1}}\ar[dd]^{g}\\
X &  &  &  &  &  & Z\\
 & B\ar[lu]^{s_{2}} &  & Q\ar[ll]^{u_{2}'}\ar[rr]_{t_{2}'} &  & N\ar[ru]_{v_{2}}
}
\]
The horizontal and vertical composition laws in addition to the interchange
property give the necessary coherence. In the subpullback case coherence
is automatic from uniqueness of mediating morphisms just as in the
standard setting.
\end{proof}

These properties of bicategory structures on span may be summarized
simply as follows.
\begin{cor}
Strictly functorially decomposable bicategory structures on hom-categories
of spans over $\mathcal{E}$ are in equivalence with locally semicartesian
category structures on $\mathcal{E}$.
\end{cor}

\begin{rem}
This means the usual bicategory structure on spans using pullbacks
is the terminal bicategory structure, and that locally subcartesian
categories give a subterminal bicategory structure on spans.
\end{rem}

\begin{prop}
\label{spantocat} There is a canonical pseudofunctor 
\[
\mathbf{Span}\left(\mathcal{E},\otimes\right)\to\mathbf{Cat}\colon\left(s,t\right)\mapsto\Sigma_{t}\nabla_{s}
\]
\end{prop}

\begin{proof}
The pseudofunctoriality constraints are constructed from the Beck
data of Prop. \ref{beck1}, and the action on 2-cells is defined by
whiskering the counit candidate map of Remark \ref{counitmap}.
\end{proof}

\section{Locally subcartesian closed categories\label{locsubcartclosed}}

We now consider the theory of locally subcartesian categories in which
the subcartesian monoidal structure on each slice category is closed.
We will later see that these will turn out to be those for which each
affine base change $\nabla_{p}$ has a right adjoint $\boxtimes_{p}$,
and we will call this right adjoint the dependent subproduct. 
\begin{defn}
A \emph{locally subcartesian closed category} $\mathcal{E}$ is a
locally subcartesian category with pullbacks for which the subcartesian
monoidal structure $f\otimes_{X}g=\Sigma_{g}\nabla_{g}\left(f\right)$
on each slice $\mathcal{E}/X$ is right closed. The semicartesian
and symmetric variants are defined similarly.
\end{defn}

\begin{rem}
In some cases, such as the category of affine spaces, one has a similar
structure except that each indexing is a set. Namely giving a pseudofunctor
from $\mathbf{Set}^{\textnormal{op}}$ to symmetric closed monoidal
categories\footnote{This example where $\mathcal{E}$ is the category of affine spaces
is somewhat close to working, as one has a canonical $\boxtimes_{X\to\mathbf{1}}\colon\mathcal{E}/X\to\mathcal{E}$
since the morphisms in the slice are closed under affine combinations.
The issue here is with cross-fiber combinations.}. We will mostly ignore such cases here as they lack the desired polynomial
structure.
\end{rem}

In much of this section we will make use of Dubuc's adjoint triangle
theorem \cite{dubuc}, and use the following formulation given by
Street and Verity \cite{torsors}. 
\begin{prop}
\label{adjtri} Suppose a functor $P\colon\mathcal{B}\to\mathcal{C}$
has a right adjoint $Q$, such that all components of the unit $\eta\colon1\Rightarrow QP$
are regular monomorphisms in $\mathcal{B}$. Supposing $\mathcal{A}$
has equalizers of coreflexive pairs, any functor $F\colon\mathcal{A}\to\mathcal{B}$
has a right adjoint if and only if $PF$ does. 
\end{prop}

This result will be needed to establish the equivalence of monoidal
closure on slices and each affine base change $\nabla_{p}$ having
a right adjoint.
\begin{rem}
Note the additional assumption that the category $\mathcal{E}$ has
pullbacks. This requirement is largely due to our applications of
the adjoint triangle theorem requiring each dependent sum $\Sigma_{p}$
to have a right adjoint. There are other reasons that we need pullbacks,
for instance one cannot even state the correct distributivity condition
without it, and it is one of the operations needed to manipulate the
data of a polynomial\footnote{An interesting example in light of this is $\mathbf{FI}^{\textnormal{op}}$
which is locally subcartesian with dependent subproducts given by
$\boxtimes_{p}\left(f\right)=B+A\setminus_{f}E$ but is without pullbacks.}. Whilst one can attempt giving constructions such as distributivity
data using only subpullbacks, this results in only non-invertible
comparison data. We regard it as essential that this data is invertible,
so that the later bi-equivalence of polynomials and polynomial functors
respects composition. We also warn the reader that as these categories
have both pullbacks and chosen subpullbacks (and some constructions
will need both) extra care must be taken.
\end{rem}

\subsection{Distributivity subpullbacks}

Before we can understand the more abstract theory of locally subcartesian
closed categories, we will need to know the basic properties of distributivity
subpullbacks, and what it means for an affine base change $\nabla_{p}$
to have a right adjoint. We also show how these properties give rise
to Beck and distributivity morphisms involving such right adjoints. 
\begin{defn}
\label{defdspb} For a locally subcartesian category $\mathcal{E}$
with pullbacks, the \emph{distributivity subpullback} of $f$ along
$p$ is defined as the subpullback $\left(e,p',d\right)$ as below
\[
\xymatrix@=1em{L\ar[rrrr]^{s}\ar@/_{1pc}/[rdd]_{h}\ar@{..>}[rd] &  & \; & \; & M\ar@/^{1pc}/[lddd]_{v}\ar@{..>}[ld]\\
 & R\myard{p'}{rr}\ar[d]_{e}\ar@{}[rru]|-{\textnormal{spb}} &  & D\ar[dd]^{d}\\
 & A\ar[d]_{f}\\
 & E\ar[rr]_{p}\ar@{}[rruu]|-{\textnormal{dspb}} &  & B
}
\]
which is terminal in the sense that given any other subpullback $\left(h,s,v\right)$
around $f$ and $p$ there exists a unique commuting diagram as above.
We call such a $p$ \emph{subexponentiable }if all distributivity
subpullbacks over $p$ exist.
\end{defn}

We now see that subexponentiable morphisms $p$ are exactly those
for which each affine base change $\nabla_{p}$ has a right adjoint.
\begin{prop}
If $\nabla_{p}$ has a right adjoint $\boxtimes_{p}$ then all distributivity
subpullbacks along $p$ exist, and at a general $f$ is given by
\[
\xymatrix@=1em{E\otimes_{B}D\myar{\pi_{2}}{rr}\ar[d]_{\epsilon_{f}} &  & D\ar[dd]^{\boxtimes_{p}f}\\
A\ar[d]_{f}\\
E\ar[rr]_{p}\ar@{}[rruu]|-{\textnormal{dspb}} &  & B
}
\]
Conversely, if all distributivity subpullbacks along $p$ exist then
$\nabla_{p}$ has a right adjoint.
\end{prop}

\begin{proof}
This is the adjunction $\nabla_{p}\dashv\boxtimes_{p}$ described
in terms of universal arrows via the counit $\epsilon_{f}$. The assumption
that the induced square is a subpullback is redundant by Lemma \ref{subpasting}.
\end{proof}

\begin{rem}
Taking $\mathcal{E}$ to be the Lawvere quantale, for any composable
pair of morphisms ($\geq$ assertions) $f\colon A\to E$ and $p\colon E\to B$
the dependent subproduct is given by $\text{\ensuremath{\boxtimes_{p}\left(f\right)}}\colon A-E+B\to B$.
Note the vertex of the subpullback of $p$ and $\boxtimes_{p}\left(f\right)$
is $A$, different from the pullback which is $\textnormal{max}\left(E,A-E+B\right)$,
so we see that distributivity subpullbacks are not pullbacks in $\mathcal{E}$
in general.
\end{rem}

\begin{rem}
Similar to the case of subpullbacks, we assume without loss of generality
that identity maps are preserved; that is $\boxtimes_{p}\left(\textnormal{id}_{E}\right)=\textnormal{id}_{B}$.
\end{rem}

\begin{prop}
Suppose $\mathcal{E}$ is a locally subcartesian category with pullbacks.
Then subexponentiable maps are subpullback stable.
\end{prop}

\begin{proof}
For a subpullback square of a cospan $p$ and $f$ we have the Beck
isomorphism $\Sigma_{f'}\nabla_{p'}\cong\nabla_{p}\Sigma_{f}$ by
Lemma \ref{beck1} and if $p$ is subexponentiable the right hand
side has a right adjoint and so by Prop. \ref{adjtri} so does $\nabla_{p'}$.
Note that unit components of adjunctions of the form $\Sigma_{f}\dashv\Delta_{f}$
are regular monomorphisms in the relevant slice category (which has
all finite limits).
\end{proof}

It is worth noting there is another characterization of subexponentiable
maps based on an analogue of ``powerful'' morphisms.
\begin{prop}
For any morphism $p\colon E\to B$ in a locally subcartesian category
$\mathcal{E}$ with finite limits, the functor $\nabla_{p}\colon\mathcal{E}/B\to\mathcal{E}/E$
has a right adjoint if and only if $\Sigma_{!_{E}}\nabla_{p}\colon\mathcal{E}/B\to\mathcal{E}$
does.
\end{prop}

\begin{proof}
This follows by the same argument as \cite[Prop. 2.4]{torsors} but
with $\Sigma_{!_{E}}\nabla_{p}=\left(-\right)\otimes_{B}E$ the apex
of the subcartesian square. 
\end{proof}

\subsection{Calibrating spans with subpullbacks}

It was noticed by Street \cite{polyspans} that distributivity pullbacks
may be seen as bipullbacks in bicategories of spans (with reversed
2-cells) over a locally cartesian closed $\mathcal{E}$. Similar observations
also appear in literature under the alternative name of ``lax pullback
complements'' using pullbacks in a 1-category of spans \cite{laxpbcomp}.
We now give the analogue for locally subcartesian closed $\mathcal{E}$. 

Here a bipullback in a bicategory refers to the pseudo-terminal pseudo-commuting
square over a cospan. Note this definition does not directly involve
lax commuting squares; instead the property can be later upgraded
to the lax setting by requiring one leg to be a groupoid opfibration.
\begin{lem}
Suppose $\mathcal{E}$ is locally subcartesian closed. In $\mathbf{Span}\left(\mathcal{E},\otimes\right)$
the bipullback of spans $\left(1,f\right)$ and $\left(1,g\right)$
is given by the left square below 
\[
\xymatrix@=1em{P\ar@{}[rr]|-{/}\ar[rr]^{\left(1,g'\right)}\ar@{}[dd]|-{/}\ar[dd]_{\left(1,f'\right)} &  & X\ar@{}[dd]|-{/}\ar[dd]^{\left(1,f\right)} &  &  &  & P & P\ar[r]^{g'}\ar[l]_{1} & X\\
 &  &  &  & : &  & P\ar[d]_{f'}\ar[u]^{1} &  & X\ar[d]^{f}\ar[u]_{1}\\
E\ar@{}[rr]|-{/}\ar[rr]_{\left(1,g\right)} &  & I &  &  &  & E & E\ar[r]_{g}\ar[l]^{1} & I
}
\]
where $f'$ and $g'$ are constructed by taking the underlying pullback
in $\mathcal{E}$ and the right describes the underlying span data.
\end{lem}

\begin{proof}
This property that the canonical map $f\mapsto\left(1,f\right)$ sends
pullback in $\mathcal{E}$ to bipullbacks in $\mathbf{Span}\left(\mathcal{E},\otimes\right)$
is a trivial exercise since it only needs composition of spans of
the form $\left(s,t\right);\left(1,f\right)$.
\end{proof}

Since the above bipullback is built from spans of the form $\left(1,f\right)$
the non-trivial subpullback structure is never used. In contrast,
the following lemma does require the subpullback structure (and the
resulting bipullback is terminal with respect to this structure).
\begin{lem}
Suppose $\mathcal{E}$ is locally subcartesian closed. In $\mathbf{Span}\left(\mathcal{E},\otimes\right)$
the bipullback of spans $\left(1,f\right)$ and $\left(p,1\right)$
is given by the left square below
\[
\xymatrix@=1em{D\ar@{}[rr]|-{/}\ar[rr]^{\left(p',e\right)}\ar@{}[dd]|-{/}\ar[dd]_{\left(1,d\right)} &  & X\ar@{}[dd]|-{/}\ar[dd]^{\left(1,f\right)} &  &  &  & D & R\ar[r]^{e}\ar[l]_{p'} & X\\
 &  &  &  & : &  & D\ar[d]_{d}\ar[u]^{1} &  & X\ar[d]^{f}\ar[u]_{1}\\
B\ar@{}[rr]|-{/}\ar[rr]_{\left(p,1\right)} &  & E &  &  &  & B & E\ar[r]_{1}\ar[l]^{p} & E
}
\]
where $e,d$ and $p'$ are constructed by forming the distributivity
subpullback as in Definition \ref{defdspb} and the right describes
the underlying span data.
\end{lem}

\begin{proof}
This is mostly the same argument as given in the case of usual distributivity
pullbacks by Street \cite{polyspans}.
\end{proof}

The description of the bicategory of polynomials for situations more
general than a plain locally cartesian closed category is given by
Street in the context of a protocalibration \cite{polyspans,proto}.
Our situation is indeed more general and so it will be convenient
to use this theory here. However, since our situation is only slightly
more general in that we are still using spans just with a more general
composition, the resulting polynomial calculus will appear quite similar.
\begin{defn}
\cite[Def. 1.2]{proto} A \emph{protocalibration} on a bicategory
$\mathscr{C}$ is a set $\mathscr{R}$ of 1-cells such that (1) $\mathscr{R}$
contains equivalences, and is closed under invertible 2-cells; (2)
$\mathscr{R}$ is closed under 1-cell composition; (3) the bicategorical
pullback of any $f\in\mathscr{R}$ along $g\in\mathscr{C}$ exists
and lies in $\mathscr{R}$; (4) every 1-cell in $\mathscr{R}$ is
a groupoid opfibration.
\end{defn}

We now use the results of this section to show bipullbacks exist (and
are stable) in our subcartesian bicategory of spans with one leg in
the restricted class of maps of the form $\left(1,f\right)$, thus
showing we have such a protocalibration.
\begin{cor}
\label{chosenproto} The replete image $\mathscr{R}$ of the canonical
pseudofunctor 
\[
\mathcal{E}\to\mathbf{Span}\left(\mathcal{E},\otimes\right)^{\textnormal{co}}\colon f\mapsto\left(1,f\right)
\]
is a protocalibration\footnote{There is an argument here for instead using $\mathbf{Span}\left(\mathcal{E},\otimes\right)^{\textnormal{coop}}$
with 1-cells also reversed, based on the Kleisli viewpoint. However,
we will stay with the usual version to match the conventions used
by Street.}.
\end{cor}

\begin{proof}
In the bicategory $\mathbf{Span}\left(\mathcal{E},\otimes\right)$
it is easy to see the operation of composing with a span $\left(1,v\right)\colon Y\nrightarrow Z$
defining the functor 
\[
\mathbf{Span}\left(\mathcal{E},\otimes\right)\left(X,Y\right)\to\mathbf{Span}\left(\mathcal{E},\otimes\right)\left(X,Z\right)
\]
is a groupoid fibration. The proof is the same as the usual setting,
using only that identity maps are respected. We also note this corresponds
to a groupoid opfibration when one reverses 2-cells. We construct
the bipullback $\left(1,f\right)$ along any general span $\left(p,g\right)$
by decomposing it as $\left(p,1\right);\left(1,g\right)$, noting
the class $\mathscr{R}$ is respected in both bipullbacks.
\end{proof}

\subsection{Distributivity and composition}

Similar to how subpullbacks satisfy the pasting lemma of Prop. \ref{subpasting},
distributivity subpullbacks also satisfy a number of compositional
properties. Here there are three versions: horizontal pastings, vertical
pastings, and forming cubes. These respectively correspond to pseudofunctoriality
of the dependent subproduct, distributivity isomorphisms (a subcartesian
type theoretic axiom of choice), and the second Beck condition (a
subcartesian Frobenius condition). 

In a plain locally cartesian category one always has the cartesian
version of these isomorphisms. The fact that we have the analogous
isomorphisms in our setting acts as a further sanity check on the
validity of the main structure of this paper. Moreover, these isomorphisms
will be an essential component of the polynomial theory just as in
the cartesian setting \cite{WalkerPoly}.

We now show these properties using the words horizontal and vertical
here in reference to how the diagrams appear in the polynomial calculus,
a different convention from \cite{laxpbcomp}. We will use the labeling
convention of Weber \cite{weber} for the horizontal and cube versions
for easy comparison.
\begin{lem}
[Horizontal distributivity subpullback lemma]\label{hordistsubpasting}
For a locally subcartesian category $\mathcal{E}$, and given diagram
as below where the bottom left square is a distributivity subpullback
and the remaining two squares are subpullbacks,
\[
\xymatrix@=1em{B_{6}\ar[rr]^{h_{9}}\ar[d]_{h_{8}} &  & B_{4}\ar[d]_{h_{5}}\ar[rr]^{h_{6}} &  & B_{5}\ar[ddd]^{h_{7}}\\
B_{2}\myard{h_{3}}{rr}\ar[d]_{h_{2}}\ar@{}[rru]|-{\textnormal{spb}} &  & B_{3}\ar[dd]^{h_{4}}\\
B\ar[d]_{h}\\
X\ar[rr]_{f}\ar@{}[rruu]|-{\textnormal{dspb}} &  & Y\ar[rr]_{g}\ar@{}[rruuu]|-{\textnormal{(d)spb}} &  & Z
}
\]
the rightmost square is a distributivity subpullback if and only if
the outside is one of $h$ along $gf$.
\end{lem}

\begin{proof}
The outside of the diagram below is a bipullback
\[
\xymatrix@=0.7em{B_{5}\ar@{}[dd]|-{/}\ar[dd]_{\left(1,h_{7}\right)}\ar@{}[rrrr]|-{/}\ar[rrrr]^{\left(h_{6},h_{5}\right)}\ar@{}[rrrrdd]|-{\textnormal{(bpb)}} &  &  &  & B_{3}\ar@{}[dd]|-{/}\ar[dd]^{\left(1,h_{4}\right)}\ar@{}[rrrr]|-{/}\ar[rrrr]^{\left(h_{3},h_{2}\right)}\ar@{}[rrrrdd]|-{\quad\;\textnormal{bpb}} &  &  &  & B\ar@{}[dd]|-{/}\ar[dd]^{\left(1,h\right)}\\
\\Z\ar@{}[rrrr]|-{/}\ar[rrrr]_{\left(g,1\right)} &  &  &  & Y\ar@{}[rrrr]|-{/}\ar[rrrr]_{\left(f,1\right)} &  &  &  & X
}
\]
if and only if the left square is.
\end{proof}

We now give the vertical version of this composition, where the reader
will note the need for the rotated square to be a pullback. Indeed,
this is one of the main reasons for needing pullbacks in our theory.
\begin{lem}
[Vertical distributivity subpullback lemma]\label{vertdistsubpasting}
For a locally subcartesian category $\mathcal{E}$, and given diagram
as on the left below where the indicated left square is a pullback,
the bottom square is a distributivity subpullback and the top square
is a subpullback
\[
\xymatrix@=1em{ & X\ar[d]_{u}\ar[rr]^{p''} &  & T\ar[dd]^{v}\\
 & B\ar[dr]_{h'}\ar[dl]_{e'}\\
R\ar[rd]_{h}\ar@{}[rr]|-{\textnormal{pb}\;\;\;\;} & \ar@{}[rruu]|-{\textnormal{\;\;(d)spb}} & S\ar[dl]_{e}\ar[r]^{p'} & D\ar[dd]^{d}\\
 & A\ar[d]_{f}\\
 & E\ar[rr]_{p}\ar@{}[rruu]|-{\textnormal{\;\;dspb}} &  & B
}
\]
the top square is a distributivity subpullback if and only if the
outside diagram is one of $fh$ along $p$.
\end{lem}

\begin{proof}
The outside of the diagram below is a bipullback
\[
\xymatrix@=0.7em{T\ar@{}[rrrr]|-{/}\ar[rrrr]^{\left(p'',u\right)}\ar@{}[dd]|-{/}\ar[dd]_{\left(1,v\right)}\ar@{}[rrrrdd]|-{\textnormal{(bpb)}} &  &  &  & B\ar@{}[dd]|-{/}\ar[dd]^{\left(1,h'\right)}\ar@{}[rrrr]|-{/}\ar[rrrr]^{\left(1,e'\right)}\ar@{}[rrrrdd]|-{\quad\;\textnormal{bpb}} &  &  &  & R\ar@{}[dd]|-{/}\ar[dd]^{\left(1,h\right)}\\
\\D\ar@{}[dd]|-{/}\ar[dd]_{\left(1,d\right)}\ar@{}[rrrr]|-{/}\ar[rrrr]_{\left(p',1\right)}\ar@{}[rrrrrrrdd]|-{\textnormal{bpb}} &  &  &  & S\ar@{}[rrrr]|-{/}\ar[rrrr]_{\left(1,e\right)} &  &  &  & A\ar@{}[dd]|-{/}\ar[dd]^{\left(1,f\right)}\\
\\B\ar@{}[rrrrrrrr]|-{/}\ar[rrrrrrrr]_{\left(p,1\right)} &  &  &  &  &  &  &  & E
}
\]
if and only if the top rectangle and equivalently top left square
is.
\end{proof}

This allows us to state the subcartesian type theoretic axiom of choice.
We warn the reader that this requires pullbacks since the above does.
\begin{cor}
[$\Sigma$-$\boxtimes$-$\Delta$ distributivity comparison]\label{beck2}
For a locally subcartesian closed category $\mathcal{E}$, and subpullback
square as on the left below
\[
\xymatrix@=0.7em{S\ar[rr]^{p'}\ar[d]_{e} &  & D\ar[dd]^{d} &  &  &  &  &  & \mathcal{E}/S\dtwocell[0.5]{rddr}{\mathfrak{d}}\myar{\boxtimes_{p'}}{rr} &  & \mathcal{E}/D\ar[dd]^{\Sigma_{d}}\\
A\ar[d]_{f} &  &  &  &  &  &  &  & \mathcal{E}/A\ar[d]_{\Sigma_{f}}\ar[u]^{\Delta_{e}}\\
E\ar[rr]_{p} &  & B &  &  &  &  &  & \mathcal{E}/E\myard{\boxtimes_{p}}{rr} &  & \mathcal{E}/B
}
\]
there is a canonical natural transformation $\mathfrak{d}$ as on
the right above. Moreover, this transformation is invertible if and
only if the left square is a distributivity subpullback of $f$ along
$p$.
\end{cor}

\begin{proof}
For a general $h\colon R\to A$ we may form the pullback $\Delta_{e}\left(h\right)=h'$
and apply $\boxtimes_{p'}$ which is equivalent to placing a distributivity
subpullback on top. Clearly the other process of composing $h$ with
$f$ and applying $\boxtimes_{p}$ results in the distributivity subpullback
of $fh$ along $p$, and so the earlier construction has an induced
map $\mathfrak{d}_{h}$ into it. If the starting left square is a
distributivity subpullback then this map is invertible by Lemma \ref{vertdistsubpasting},
and the converse is recovered by taking $h$ to be the identity.
\end{proof}

We now state the subcartesian version of the cube lemma, which will
also require both pullbacks and subpullbacks. 
\begin{lem}
[Cube distributivity subpullback lemma]\label{cubedistsubpasting}
For a locally subcartesian closed category $\mathcal{E}$, and given
cube as below where the front is a distributivity subpullback, the
rear is a subpullback, the bottom and top are subpullbacks, and the
lower left square is a pullback
\[
\xymatrix@=1em{ & A_{1}\ar[rrr]^{f_{1}}\ar[dl]_{h_{1}}\ar@{..>}[d]^{d_{1}} &  &  & B_{1}\ar[dd]^{d_{5}}\ar[dl]^{k_{1}}\\
C_{1}\ar[d]_{d_{3}}\ar@/^{0.5pc}/[rrr]^{\quad\quad g_{1}}\ar@{}[r]|-{} & A_{3}\ar@{..>}[ld]\sb(.35){h_{3}}\ar@{..>}[d]^{d_{2}} &  & D_{1}\ar@/_{0pc}/[dd]\sb(.7){d_{6}}\ar@{}[rd]|-{\textnormal{}}\\
C_{3}\ar[d]_{d_{4}} & A_{2}\ar@{..>}[ld]\sb(.35){h_{2}}\ar@{..>}[rrr]^{f_{2}\quad\quad} &  &  & B_{2}\ar[dl]^{k_{2}}\\
C_{2}\ar[rrr]_{g_{2}} &  &  & D_{2}
}
\]
the rear face is a distributivity subpullback if and only if the right
face is a pullback. 
\end{lem}

\begin{proof}
Considering the unlabeled regions to be just commuting, the outside
of the two diagrams (which are equal) below is a bipullback precisely
when
\[
\xymatrix@=0.7em{ &  &  & A_{1}\ar@{}[rrd]|-{/}\ar[rrd]\sb(.25){\left(1,h_{1}\right)}\ar@{}[rr]|-{/}\ar[rr]^{\left(1,d_{1}\right)} &  & A_{3}\ar@{}[rrd]|-{/}\ar[rrd]^{\left(1,h_{3}\right)} &  &  &  &  &  & A_{1}\ar@{}[rr]|-{/}\ar[rr]^{\left(1,d_{1}\right)} &  & A_{3}\ar@{}[dd]|-{/}\ar[dd]^{\left(1,d_{2}\right)}\ar@{}[rrdr]|-{/}\ar[rrdr]^{\left(1,h_{3}\right)}\\
B_{1}\ar@{}[rrr]|-{/}\ar[rrr]_{\left(1,k_{1}\right)}\ar@{}[dd]|-{/}\ar[dd]_{\left(1,d_{5}\right)}\ar@{}[rrur]|-{/}\ar[urrr]^{\left(f_{1},1\right)}\ar@{}[rrrdd]|-{\textnormal{(bpb)}} &  &  & D_{1}\ar@{}[rr]|-{/}\ar[rr]_{\left(g_{1},1\right)}\ar@{}[dd]|-{/}\ar[dd]^{\left(1,d_{6}\right)}\ar@{}[rrrrdd]|-{\quad\textnormal{bpb}} &  & C_{1}\ar@{}[rr]|-{/}\ar[rr]_{\left(1,d_{3}\right)} &  & C_{3}\ar@{}[dd]|-{/}\ar[dd]^{\left(1,d_{4}\right)} &  & B_{1}\ar@{}[rur]|-{/}\ar[rru]^{\left(f_{1},1\right)}\ar@{}[dd]|-{/}\ar[dd]^{\left(1,d_{5}\right)}\ar@{}[rrrr]|-{\textnormal{(bpb)}} &  &  &  & \;\ar@{}[rrdr]|-{\quad\textnormal{bpb}} &  &  & C_{3}\ar@{}[dd]|-{/}\ar[dd]^{\left(1,d_{4}\right)}\\
 &  &  &  &  &  &  &  &  &  &  &  &  & A_{2}\ar@{}[rdrr]|-{/}\ar[drrr]^{\left(1,h_{2}\right)} &  &  & \;\\
B_{2}\ar@{}[rrr]|-{/}\ar[rrr]_{\left(1,k_{2}\right)} &  &  & D_{2}\ar@{}[rrrr]|-{/}\ar[rrrr]_{\left(g_{2},1\right)} &  &  &  & C_{2} &  & B_{2}\ar@{}[rrrru]|-{/}\ar[rrrru]^{\left(f_{2},1\right)}\ar@{}[rrr]|-{/}\ar[rrr]_{\left(1,k_{2}\right)} &  &  & D_{2}\ar@{}[rrr]|-{/}\ar[rrrr]_{\left(g_{2},1\right)} &  &  &  & C_{2}
}
\]
the left square of either of the diagrams above is.
\end{proof}

We can now give the canonical morphisms describing the interaction
of pullback and the dependent subproduct.
\begin{cor}
[$\Delta$-$\boxtimes$ Beck comparison]\label{beck3} For a locally
subcartesian closed category $\mathcal{E}$, a commuting square which
factors through the chosen subpullback as on the left below gives
rise to a natural transformation $\mathfrak{t}$ as on the right below
\[
\xymatrix@=0.7em{C\ar[rr]\ar[dd]_{g'}\ar[rr]^{f'} &  & B\ar[dd]^{g} &  &  &  &  &  & \mathcal{E}/C\ar[rr]^{\boxtimes_{f'}}\utwocell[0.5]{rdrd}{\mathfrak{t}} &  & \mathcal{E}/B\\
\\A\ar[rr]_{f} &  & T &  &  &  &  &  & \mathcal{E}/A\ar[rr]_{\boxtimes_{f}}\ar[uu]^{\Delta_{g'}} &  & \mathcal{E}/T\ar[uu]_{\Delta_{g}}
}
\]
and this transformation is invertible if and only if the left square
is a subpullback.
\end{cor}

\begin{proof}
Note that the canonical transformation $\mathfrak{t}$ corresponds
to a transformation of left adjoints $\mathfrak{\mathfrak{b}}\colon\Sigma_{g'}\nabla_{f'}\Rightarrow\nabla_{f}\Sigma_{g}$
given by Prop. \ref{beck1}. Moreover, this $\mathfrak{t}$ is invertible
if and only if $\mathfrak{b}$ is, which happens if and only if the
original square is a subpullback by Prop. \ref{beck1}. To explain
why this  also follows from Lemma \ref{cubedistsubpasting}, note
there is a more involved proof where $\mathfrak{t}$ is directly constructed
via a cube construction, and this gives the same transformation $\mathfrak{t}$
as the above application of mates. This is explained by Weber \cite[Lemma 3.2.3]{weber}
in the cartesian setting.
\end{proof}

\subsection{Indexed closed monoidal structure}

Recalling that monoidal product on each slice is given by $\left(-\right)\otimes g=\Sigma_{g}\nabla_{g}\left(-\right)$,
we can now give the expected characterization of local subcartesian
closure in terms of adjointness, also applying this to the indexed
monoidal category viewpoint. This section is fairly short as we have
already developed all of the necessary theory. It is clear these facts
apply in both the semicartesian and subcartesian settings.
\begin{prop}
\label{affright} Suppose $\mathcal{E}$ is a locally subcartesian
category with pullbacks. This structure is locally subcartesian closed
if and only if each affine base change $\nabla_{p}$ has a right adjoint
$\boxtimes_{p}$.
\end{prop}

\begin{proof}
Each right tensor by an object $p$ is given by $\Sigma_{p}\nabla_{p}$
and this has a right adjoint if and only if every $\nabla_{p}$ does
by Prop. \ref{adjtri}.
\end{proof}

\begin{rem}
In this case the monoidal closure is given by $\left[f,g\right]\cong\boxtimes_{f}\Delta_{f}\left(g\right)$.
That the monoidal closure of the slices is necessarily defined this
way is simply composition of right adjoints. In the cartesian setting,
this reduces to the standard internal hom of the slice category $\Pi_{f}\Delta_{f}\left(g\right)$
which is one reason we have chosen the right closure convention.
\end{rem}

In some cases one has a locally subcartesian category where the base
structure $\mathcal{E}$ is subcartesian closed, but this does not
extend to each slice category. The following concerns this situation.
\begin{prop}
Suppose $\mathcal{E}$ is locally subcartesian with finite limits.
Then each functor $\nabla_{Y}\colon\mathcal{E}\to\mathcal{E}/Y$ has
a right adjoint $\boxtimes_{Y}\colon\mathcal{E}/Y\to\mathcal{E}$
precisely when the subcartesian structure on $\mathcal{E}\cong\mathcal{E}/\mathbf{1}$
is right closed.
\end{prop}

\begin{proof}
The subcartesian monoidal structure on $\mathcal{E}\cong\mathcal{E}/\mathbf{1}$
is given by $f\otimes g=\Sigma_{g}\nabla_{g}\left(f\right)$ where
$f$ and $g$ may be seen respectively as maps from objects $X$ and
$Y$ to the terminal. This is right closed only when $\Sigma_{g}\nabla_{g}$
has a right adjoint. By Prop. \ref{adjtri} this happens precisely
when $\nabla_{g}$ has a right adjoint.
\end{proof}

Similar to the earlier setting of locally subcartesian categories,
we see that locally subcartesian closed categories may be defined
as a special type of indexed closed monoidal category.
\begin{cor}
Suppose $\mathcal{E}$ has pullbacks. A special indexed monoidal category
$\nabla_{\left(-\right)}\colon\mathcal{E}^{\textnormal{op}}\to\mathbf{MonCat}$
as in Definition \ref{spec} is closed if and only if each $\nabla_{p}$
has a right adjoint.
\end{cor}

\begin{proof}
Such indexed monoidal categories have tensor $\left(-\right)\otimes g\cong\Sigma_{g}\nabla_{g}\left(-\right)$
by Prop. \ref{specmon} and so Prop. \ref{affright} applies.
\end{proof}

\begin{rem}
It follows that locally subcartesian closed categories $\mathcal{E}$
are closed under slicing by any object $X\in\mathcal{E}$. In particular,
for any $a\colon A\to X$ and $b\colon B\to X$, and map $f\colon A\to B$
over $X$ we have an affine base change $\nabla_{f}$ defined by the
commuting square
\[
\xymatrix@=1em{\left(\mathcal{E}/X\right)/a\ar[d]_{\cong} &  & \left(\mathcal{E}/X\right)/b\ar[d]^{\cong}\ar[ll]_{\;\nabla_{f}}\\
\mathcal{E}/A &  & \mathcal{E}/B\ar[ll]^{\;\nabla_{\textnormal{dom}\left(f\right)}}
}
\]
and similarly for $\Sigma_{f}$, $\Delta_{f},\;\boxtimes_{f}$, giving
a special indexed monoidal category on $\mathcal{E}/X$.
\end{rem}

The pullback operation does not respect the tensor structure in general,
however the closed structure is respected in the following sense.

\begin{prop}
Suppose $\mathcal{E}$ is locally subcartesian closed. Then for all
pairs $g$ and $h$ we have the Frobenius reciprocity isomorphism
$\Delta_{f}\left[g,h\right]\cong\left[\nabla_{f}\left(g\right),\Delta_{f}\left(h\right)\right]$.
\end{prop}

\begin{proof}
We construct a subpullback of $f$ and $g$ as in Cor. \ref{beck3},
thus giving
\[
\Delta_{f}\left[g,h\right]\cong\Delta_{f}\boxtimes_{g}\Delta_{g}\left(h\right)\cong\boxtimes_{g'}\Delta_{f'}\Delta_{g}\left(h\right)\cong\boxtimes_{g'}\Delta_{g'}\Delta_{f}\left(h\right)\cong\left[\nabla_{f}\left(g\right),\Delta_{f}\left(h\right)\right]
\]
by the same corollary, pseudofunctoriality of $\Delta_{\left(-\right)}$,
and the definition of $g'$.
\end{proof}

\section{Polynomials and polynomial functors\label{bicatyspanpoly}}

In this section we define the bicategory of polynomials applicable
to this situation, discuss composition of polynomials, and establish
the bi-equivalence between polynomials and polynomial functors with
a suitable notion of strength.

\subsection{Polynomials over spans with subpullbacks}

The reason for establishing the protocalibration earlier is to give
a more straightforward way to establish the existence and description
of the bicategory of polynomials for the subcartesian setting. Of
course, one may define this structure directly and verify it forms
a bicategory, but this allows us to avoid needing to check properties
such as horizontal composition of general 2-cells and coherence. Note
that because the subcartesian bicategory of spans is so similar to
the usual cartesian setting, the resulting polynomial structure remains
very similar.

The reader will notice here that the horizontal squares are subpullbacks,
and the diamond-shaped squares are pullbacks, as these two operations
on polynomial data play a different role.
\begin{prop}
\label{polycomp} Suppose $\left(\mathcal{E},\otimes\right)$ is locally
subcartesian closed. The bicategory of polynomials over the protocalibration
$\mathscr{R}$ given by the replete image of $\mathcal{E}\to\mathbf{Span}\left(\mathcal{E},\otimes\right)^{\textnormal{co}}$
is the bicategory $\mathbf{Poly}\left(\mathcal{E},\otimes\right)$
with objects of $\mathcal{E}$, 1-cells given by
\[
\xymatrix@=1em{I &  & E\ar[rr]^{p}\ar[ll]_{s} &  & B\ar[rr]^{t} &  & J}
\]
and general 2-cells given by isomorphism classes of diagrams
\[
\xymatrix@=0.7em{ &  & E\ar[rr]^{p}\ar[ldl]_{s} &  & B\ar[rdr]^{t}\ar@{=}[d]\\
I &  & P\ar[u]\ar[d]\ar[rr]\ar@{}[drr]|-{\textnormal{spb}} &  & B\ar[d] &  & J\\
 &  & M\ar[rr]_{q}\ar[llu]^{u} &  & N\ar[rru]_{v}
}
\]
The composite of two polynomials is constructed as in the diagram
\[
\xymatrix@=0.7em{ &  &  & A\ar[rrr]^{p'}\ar[ddl]_{a} & \ar@{}[dd]|-{\textnormal{spb}} &  & T\ar[d]^{e}\ar[rrr]^{w'} &  & \ar@{}[dd]|-{\textnormal{dspb}} & R\ar[ddr]^{b}\\
 &  &  &  &  &  & S\ar[rrd]^{t'}\ar[lld]_{m'}\ar@{}[dd]|-{\textnormal{pb}}\\
 &  & E\ar[rr]^{p}\ar[ldl]_{s} &  & B\ar[rrd]^{t} &  &  &  & X\ar[rr]^{w}\ar[lld]_{m} &  & Y\ar[rrd]^{n}\\
I &  &  &  &  &  & J &  &  &  &  &  & K
}
\]
where $\textnormal{spb}$ and $\textnormal{dspb}$ indicate chosen
subpullbacks and distributivity subpullbacks respectively.
\end{prop}

\begin{proof}
The derivation of the polynomial structure is essentially the same
as given for the cartesian setting by Street \cite{polyspans}. The
difference is the underlying bicategory of the protocalibration has
composition given by subpullback rather than pullback.
\end{proof}

\begin{rem}
In principle one may use this idea of taking spans of spans using
our more general span composition. At the base level one still takes
coherent subpullbacks, and at the bicategorical level giving coherent
subpullbacks amounts to giving non-universal distributivity diagrams
(which are subpullbacks in the base) which coherently compose horizontally,
vertically and via cubes. Here one still requires that the diamond-shaped
squares are pullbacks. In short, this amounts to a locally subcartesian
category with pullbacks, a dependent tensor, and coherent distributivity
data from this tensor. We call such a category \emph{locally subcartesian
tensored} but we do not study them here due to a lack of convincing
examples (and the fact their definition would involve a number of
non-trivial coherence conditions). Note this is not quite an instance
of a protocalibration, since this uses a composition of bicategorical
spans more general than bipullback. However, we still require this
structure restricts to bipullbacks when both legs are groupoid opfibrations
$\left(1,f\right)^{\textnormal{op}}$ to ensure the diamond-shaped
squares are pullbacks, so that the usual bicategory of spans with
pullbacks remains inside polynomials.
\end{rem}

The following property of polynomial composition, that a version of
Weber's description of polynomial composition \cite{weber} still
holds, is another good reason that we have required the existence
of pullbacks. This description is useful for understanding horizontal
composition of subcartesian 2-cells.
\begin{prop}
Suppose $\mathcal{E}$ is locally subcartesian closed. The composite
of two polynomials over $\mathcal{E}$ is given by the terminal diagram
of the form below
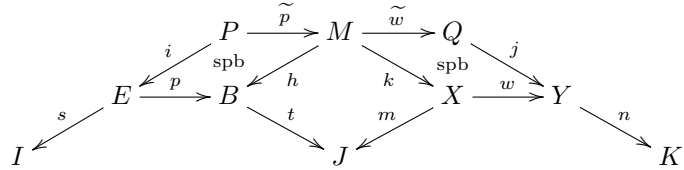
\[
\xymatrix@=1em{ &  &  &  & P\ar[rr]^{\widetilde{p}}\ar[lld]_{i}\ar@{}[d]|-{\textnormal{spb}} &  & M\ar[rr]^{\widetilde{w}}\ar[rrd]_{k}\ar[lld]^{h} &  & \ar@{}[d]|-{\textnormal{spb}}Q\ar[drr]^{j}\\
 &  & E\ar[rr]^{p}\ar[ldl]_{s} &  & B\ar[rrd]^{t} &  &  &  & X\ar[rr]^{w}\ar[lld]_{m} &  & Y\ar[rrd]^{n}\\
I &  &  &  &  &  & J &  &  &  &  &  & K
}
\]
where the indicated squares are subpullbacks and the middle square
commutes\footnote{This is actually slightly more than a polynomial composite, as it
comes equipped with a factorization of the top middle morphism (which
is unique up to unique isomorphism).}, and a morphism is a factorization of the subpullbacks of a diagram
through the subpullbacks of another.
\end{prop}

\begin{proof}
We check that the polynomial composite diagram of Prop. \ref{polycomp}
is the terminal diagram of the form above. Given any such diagram,
we first factor $h$ and $k$ via map $\epsilon\colon M\to S$ through
the pullback $S$ of $t$ and $m$. Then the subpullback $\left(\epsilon,\widetilde{w},j\right)$
around $t'$ and $w$ factors through the distributivity subpullback
$\left(e,w',b\right)$ via maps $\overline{j}\colon Q\to R$ and $\overline{k}\colon M\to T$.
We then apply Lemma \ref{spb1} to induce $\overline{i}\colon P\to A$.
Now having shown existence, the uniqueness of factorizations through
distributivity subpullbacks and uniqueness of Lemma \ref{spb1} finish
the proof of terminality.
\end{proof}

\begin{prop}
\label{polytocat} For any locally subcartesian closed category $\mathcal{E}$,
there is a canonical pseudofunctor
\[
\mathbf{Poly}\left(\mathcal{E},\otimes\right)\to\mathbf{Cat}\colon\left(s,p,t\right)\mapsto\Sigma_{t}\boxtimes_{p}\Delta_{s}
\]
\end{prop}

\begin{proof}
Abstractly this can be justified from the theory of protocalibrations
\cite{proto}, since $\mathbf{Poly}\left(\mathcal{E},\otimes\right)\simeq\left(\mathbf{Span}_{\mathscr{R}}\mathscr{C}\right)^{\textnormal{op}}$
where $\mathscr{C}=\mathbf{Span}\left(\mathcal{E},\otimes\right)^{\textnormal{co}}$
and $\mathscr{R}$ is the class of 1-cells $\left(1,f\right)$ in
$\mathscr{C}$ and those isomorphic, and so pseudofunctors $\mathbf{Poly}\left(\mathcal{E},\otimes\right)\to\mathbf{Cat}$
correspond to pseudofunctors $\mathbb{X}\colon\mathscr{C}^{\textnormal{op}}\to\mathbf{Cat}$
such that every $\mathbb{X}\left(\left(1,t\right)^{\textnormal{op}}\right)$
has a left adjoint respecting the Beck condition with pullbacks. The
canonical $\mathbb{X}\colon\mathbf{Span}\left(\mathcal{E},\otimes\right)^{\textnormal{coop}}\to\mathbf{Cat}$
sends each span $\left(p,s\right)^{\textnormal{op}}$ to $\boxtimes_{p}\Delta_{s}$
(a right adjoint analogue of Prop. \ref{spantocat}) and each $\mathbb{X}\left(\left(1,t\right)^{\textnormal{op}}\right)=\Delta_{t}$
has a left adjoint $\Sigma_{t}$ which respects the Beck condition
with pullbacks. This $\mathscr{R}$-cocompleteness then allows the
extension from $\mathscr{C}$ back to $\mathbf{Span}_{\mathscr{R}}\mathscr{C}$
showing the assignation has the above form on polynomial 1-cells,
and on each 2-cell $\left[k,\sigma\right]$ it is given by 
\[
\xymatrix@=1em{ &  & B\ar[rrd]^{\left(p,s\right)}\ar[lld]_{\left(1,t\right)}\ar[dd]|-{\left(1,k\right)} &  &  &  &  &  &  &  & XB\ar[rrdd]\sp(.6){\mathbb{X}_{*}\left(1,k\right)}\ar@{}[dd]|-{\Downarrow\epsilon_{\left(1,k\right)}}\ar@/^{0.8pc}/[rrrrdd]^{\mathbb{X}_{*}\left(1,t\right)}\\
J &  & \;\ar@{}[rr]|-{\Downarrow\sigma}\ar@{}[ll]|-{\;=} &  & I & \mapsto &  & \ar@{}[rd]|-{\Downarrow\mathbb{X}\sigma} &  &  &  &  &  & \ar@{}[ld]|-{=}\\
 &  & N\ar[urr]_{\left(q,u\right)}\ar[llu]^{\left(1,v\right)} &  &  &  & \mathbb{X}I\ar[rr]_{\mathbb{X}\left(q,u\right)}\ar@/^{0.8pc}/[rrrruu]^{\mathbb{X}\left(p,s\right)} &  & \mathbb{X}N\ar[rruu]\sp(.4){\mathbb{X}\left(1,k\right)}\ar[rrrr]_{\textnormal{id}} &  & \; &  & \mathbb{X}N\ar[rr]_{\mathbb{X}_{*}\left(1,v\right)} &  & \mathbb{X}J
}
\]
where we have chosen $\Delta_{\textnormal{id}}$ and $\boxtimes_{\textnormal{id}}$
to be identities and each $\mathbb{X}_{*}\left(1,f\right)=\Sigma_{f}$. 

We now also give a more concrete description of this pseudofunctor.
We note that in a similar way to the cartesian setting, the pseudofunctoriality
constraints are built from the $\Sigma$-$\Delta$ Beck data (with
pullbacks, not that of Lemma \ref{beck1}), the $\Delta$-$\boxtimes$
Beck data of Cor. \ref{beck3}, and $\Sigma$-$\boxtimes$-$\Delta$
distributivity data of Cor. \ref{beck2}. Note these all only involve
pullback functors $\Delta_{f}$, even though the coherence data of
the last two is constructed using subpullbacks.

In a similar fashion to Gambino and Kock's description of this pseudofunctor's
action on 2-cells in the cartesian setting \cite{gambinokock}, we
describe the action on 2-cells by assigning each vertical morphism\footnote{One may use this vertical assignment to recover $\mathbf{Span}\left(\mathcal{E}\right)^{\textnormal{co}}\to\mathbf{Cat}$.}
\[
\xymatrix@=1em{ &  & E\ar[rr]^{p}\ar[ldl]_{s} &  & B\ar[rdr]^{t}\ar@{=}[d]\\
I &  & P\ar[u]_{j}\ar[ll]^{a}\ar[rr]_{w} &  & B\ar[rr]_{t} &  & J
}
\]
to the pasting diagram 
\[
\xymatrix@=1em{ & \mathcal{E}/E\ar[rr]^{\textnormal{id}}\ar[rd]_{\Delta_{j}}\ar@{}[d]|-{\cong} & \ar@{}[dr]|-{\Downarrow\textnormal{inc}} & \mathcal{E}/E\ar[rr]^{\textnormal{id}}\ar[rd]^{\nabla_{j}} & \ar@{}[d]|-{\Downarrow\eta} & \mathcal{E}/E\ar[rd]^{\boxtimes_{p}}\ar@{}[d]|-{\cong}\\
\mathcal{E}/I\ar[rr]_{\Delta_{a}}\ar[ru]^{\Delta_{s}} & \; & \mathcal{E}/P\ar[rr]_{\textnormal{id}} & \; & \mathcal{E}/P\ar[rr]_{\boxtimes_{w}}\ar[ru]^{\boxtimes_{j}} & \; & \mathcal{E}/B\ar[rr]_{\Sigma_{t}} &  & \mathcal{E}/J
}
\]
and assigning each subcartesian morphism
\[
\xymatrix@=1em{I &  & P\ar[d]_{h}\ar[rr]\ar@{}[drr]|-{\textnormal{spb}}\ar[ll]_{a}\ar[rr]^{w} &  & B\ar[d]^{k}\ar[rr]^{t} &  & J\\
 &  & M\ar[rr]_{q}\ar[llu]^{u} &  & N\ar[rru]_{v}
}
\]
to the pasting diagram
\[
\xymatrix@=1em{\mathcal{E}/I\ar[rr]^{\Delta_{a}}\ar[rd]_{\Delta_{u}} & \ar@{}[d]|-{\cong} & \mathcal{E}/P\ar[rr]\ar@{}[dr]|-{\cong}\ar[rr]^{\boxtimes_{w}} &  & \mathcal{E}/B\ar[rr]^{\Sigma_{t}}\ar[rd]^{\Sigma_{k}}\ar@{}[d]|-{\Downarrow\epsilon} &  & \mathcal{E}/J\\
 & \mathcal{E}/M\ar[rr]_{\boxtimes_{q}}\ar[ru]_{\Delta_{h}} &  & \mathcal{E}/N\ar[ur]^{\Delta_{k}}\ar[rr]_{\textnormal{id}} & \; & \mathcal{E}/N\ar[ur]_{\Sigma_{v}}
}
\]
using again the $\Delta$-$\boxtimes$ Beck data of Cor. \ref{beck3}.
\end{proof}

\begin{rem}
In the usual bicategory of polynomials, one may view $\Sigma_{f}$,
$\Delta_{f}$ and $\Pi_{f}$ as basic polynomial diagrams, giving
adjoint triples $\Sigma_{f}\dashv\Delta_{f}\dashv\Pi_{f}$ \cite[Prop. 25]{WalkerPoly}.
However, within our more general $\mathbf{Poly}\left(\mathcal{E},\otimes\right)$
we have only the adjunction $\Sigma_{f}\dashv\Delta_{f}$, there is
no adjunction $\Delta_{f}\dashv\boxtimes_{f}$ as typically diagonals
fail to exist in semicartesian or subcartesian settings.
\end{rem}

\subsection{Subcartesian polynomial functors}

In order to understand subcartesian polynomial functors and their
equivalence with polynomial diagrams, we will first need to define
the canonical strength and enrichment structure. We make the observation
that for any locally subcartesian closed category $\mathcal{E}$ with
finite limits, each slice category $\mathcal{E}/X$ is canonically
enriched over $\mathcal{E}$. Similar to the situation of Gambino
and Kock \cite[Section 1.3]{gambinokock} for any $a\colon A\to X$,
$b\colon B\to X$ and $!_{X}\colon X\to\mathbf{1}$ we have a canonical
enrichment given by
\[
\mathcal{E}/X\left(a,b\right):=\boxtimes_{!_{X}}\boxtimes_{a}\Delta_{a}\left(b\right)\in\mathcal{E}
\]
In order for polynomial functors and their transformations to be well
behaved, they must respect this canonical enrichment. This enrichment
can equivalently be described in terms of strength, which is defined
as follows.
\begin{defn}
Suppose $\mathcal{A}$ and $\mathcal{B}$ are categories tensored
over a monoidal category $\left(\mathcal{V},\otimes,I\right)$. A
left strength on a functor $F\colon\mathcal{A}\to\mathcal{B}$ is
a family of maps $\textnormal{st}^{F}_{K,a}\colon K\circledast F\left(a\right)\to F\left(K\circledast a\right)$
natural over $K\in\mathcal{V}$ and $a\in\mathcal{A}$, and respecting
the left unitor and associator of $\mathcal{V}$. A natural transformation
$\alpha\colon F\Rightarrow G$ is $\mathcal{V}$-strong when it renders
commutative
\[
\xymatrix@=1em{K\circledast F\left(a\right)\ar[rrr]^{\textnormal{st}^{F}_{K,a}}\ar[d]_{K\circledast\alpha_{a}} &  &  & F\left(K\circledast a\right)\ar[d]^{\alpha_{K\circledast a}}\\
K\circledast G\left(a\right)\ar[rrr]_{\textnormal{st}^{G}_{K,a}} &  &  & G\left(K\circledast a\right)
}
\]
for all $K\in\mathcal{V}$ and $a\in\mathcal{A}$.
\end{defn}

In our situation, each slice category $\mathcal{E}/X$ is tensored
over $\mathcal{E}$ via the left adjoint
\[
K\circledast a=\Sigma_{a}\nabla_{a}\nabla_{!_{X}}\left(K\right)\in\mathcal{E}/X
\]
which is simply $a\cdot\textnormal{pr}_{A}\colon K\otimes_{\mathbf{1}}A\to A\to X$.
Moreover, respecting the enrichment is equivalent to respecting the
strength; in fact, this correspondence between strength and enrichment
yields an isomorphism of 2-categories \cite[Section 3]{strengththesis}.

We now exhibit the canonical notion of strength on the basic data
of subcartesian polynomial functors against this subcartesian tensor,
as well as strength against the usual cartesian tensor $K\times a=\Sigma_{a}\Delta_{a}\Delta_{!_{X}}\left(K\right)\in\mathcal{E}/X$.
\begin{prop}
Suppose $\left(\mathcal{E},\otimes\right)$ is a locally subcartesian
closed category with finite limits. Then the functors $\Sigma_{p}$,
$\nabla_{p}$, $\Delta_{p}$ and $\boxtimes_{p}$ admit a canonical
subcartesian strength. Moreover, $\Sigma_{p}$, $\Delta_{p}$ and
$\boxtimes_{p}$ admit a canonical cartesian strength.
\end{prop}

\begin{proof}
Consider maps $p\colon E\to B$, $f\colon A\to E$ and $c\colon C\to B$.
We have canonical subcartesian left strength maps for $\Sigma_{p}$
given by 
\[
K\circledast\Sigma_{p}\left(f\right)=\Sigma_{pf}\nabla_{pf}\nabla_{!_{B}}\left(K\right)\cong\Sigma_{p}\Sigma_{f}\nabla_{f}\nabla_{!_{E}}\left(K\right)=\Sigma_{p}\left(K\circledast f\right)
\]
for $\Delta_{p}$ as on the left below (where $p'$ and $c'$ are
from the pullback of $p$ and $c$)
\[
\begin{aligned}K\circledast\Delta_{p}\left(c\right) & =\Sigma_{c'}\nabla_{c'}\nabla_{!_{E}}\left(K\right) &  &  & K\circledast\boxtimes_{p}\left(f\right) & =\Sigma_{d}\nabla_{d}\nabla_{!_{B}}\left(K\right)\\
 & \cong\Sigma_{c'}\nabla_{c'}\nabla_{p}\nabla_{!_{B}}\left(K\right) &  &  &  & \overset{\eta}{\to}\Sigma_{d}\boxtimes_{p'}\nabla_{p'}\nabla_{d}\nabla_{!_{B}}\left(K\right)\\
 & \cong\Sigma_{c'}\nabla_{p'}\nabla_{c}\nabla_{!_{B}}\left(K\right) &  &  &  & \cong\Sigma_{d}\boxtimes_{p'}\nabla_{e}\nabla_{f}\nabla_{!_{E}}\left(K\right)\\
 & \overset{\textnormal{inc}}{\to}\Sigma_{c'}\Delta_{p'}\nabla_{c}\nabla_{!_{B}}\left(K\right) &  &  &  & \overset{\textnormal{inc}}{\to}\Sigma_{d}\boxtimes_{p'}\Delta_{e}\nabla_{f}\nabla_{!_{E}}\left(K\right)\\
 & \cong\Delta_{p}\Sigma_{c}\nabla_{c}\nabla_{!_{B}}\left(K\right) &  &  &  & \cong\boxtimes_{p}\Sigma_{f}\nabla_{f}\nabla_{!_{E}}\left(K\right)\\
 & =\Delta_{p}\left(K\circledast c\right) &  &  &  & =\boxtimes_{p}\left(K\circledast f\right)
\end{aligned}
\]
and the strength of $\boxtimes_{p}$ is constructed as on the right
above, where the distributivity subpullback of $f$ along $p$ is
given as in Def. \ref{defdspb}. The strength of $\nabla_{p}$ is
given as for $\Delta_{p}$ but without needing the inclusion $\nabla_{p}\hookrightarrow\Delta_{p}$
and where $p'$ and $c'$ are instead constructed from the subpullback.
Note only $\Sigma_{p}$ and $\nabla_{p}$ have invertible $\circledast$-strength
maps in general. Note also the inclusions $\nabla_{f}\Rightarrow\Delta_{f}$
are $\circledast$-strong.

It is shown by Gambino and Kock that each $\Sigma_{p}$ and $\Delta_{p}$
admit a canonical cartesian strength \cite[Section 1.3]{gambinokock}.
Unfortunately, this only gives left costrength maps $\nabla_{p}\left(K\times c\right)\to K\times\nabla_{p}\left(c\right)$
on each $\nabla_{p}$ constructed from the projection maps to $K$
and $\nabla_{p}\left(c\right)$ in the below diagram
\[
\xymatrix@=0.7em{\nabla_{p}\left(K\times_{\mathbf{1}}c\right)\ar[rr]\ar[d]\ar@{}[drr]|(.55){\textnormal{spb}} &  & K\times_{\mathbf{1}}C\ar[rr]\ar[d]\ar@{}[drr]|-{\textnormal{pb}} &  & K\ar[d]\\
\nabla_{p}\left(c\right)\ar[rr]\ar[d]\ar@{}[drr]|-{\textnormal{spb}} & \; & C\ar[d]^{c}\ar[rr] &  & \mathbf{1}\\
E\ar[rr]_{p} &  & B
}
\]
However, we have a left strength on $\boxtimes_{p}$ with components
$K\times\boxtimes_{p}\left(f\right)\to\boxtimes_{p}\left(K\times f\right)$
constructed as the adjunct of $\nabla_{p}\left(K\times\boxtimes_{p}\left(f\right)\right)\to K\times\nabla_{p}\boxtimes_{p}\left(f\right)\to K\times f$. 
\end{proof}

\begin{defn}
We call functors and transformations \emph{bunched strong} when they
respect both the cartesian and subcartesian strength.
\end{defn}

\begin{example}
The two resulting transformations of Prop. \ref{polytocat} are both
cartesian and subcartesian strong. The issue that the unit $\eta\colon1\Rightarrow\boxtimes_{j}\nabla_{j}$
is not cartesian strong is resolved upon pasting with the inclusion
$\nabla_{j}\hookrightarrow\Delta_{j}.$
\end{example}

\begin{rem}
The subcartesian and cartesian strengths defined above interact coherently
via the inclusions $\iota_{K,a}\colon K\circledast a\to K\times a$.
For any transformation $\alpha\colon F\Rightarrow G$ this gives the
equality of the outer rectangle of 
\[
\xymatrix@=1em{K\circledast F\left(a\right)\ar[rr]^{\iota_{K,F\left(a\right)}}\ar[d]_{K\circledast\alpha_{a}} &  & K\times F\left(a\right)\ar[rr]^{\textnormal{st}^{F,\times}_{K,a}}\ar[d]^{K\times\alpha_{a}} &  & F\left(K\times a\right)\ar[d]^{\alpha_{K\times a}}\\
K\circledast G\left(a\right)\ar[rr]_{\iota_{K,G\left(a\right)}} &  & K\times G\left(a\right)\ar[rr]_{\textnormal{st}^{G,\times}_{K,a}} &  & G\left(K\times a\right)
}
\]
and the outer rectangle of 
\[
\xymatrix@=1em{K\circledast F\left(a\right)\ar[rr]^{\textnormal{st}^{F,\circledast}_{K,a}}\ar[d]_{K\circledast\alpha_{a}} &  & F\left(K\circledast a\right)\ar[d]^{\alpha_{K\circledast a}}\ar[rr]^{F\left(\iota_{K,a}\right)} &  & F\left(K\times a\right)\ar[d]^{\alpha_{K\times a}}\\
K\circledast G\left(a\right)\ar[rr]_{\textnormal{st}^{G,\circledast}_{K,a}} &  & G\left(K\circledast a\right)\ar[rr]_{G\left(\iota_{K,a}\right)} &  & G\left(K\times a\right)
}
\]
It follows that if $G$ preserves monomorphisms, then any transformation
$\alpha\colon F\Rightarrow G$ that respects cartesian strength also
respects subcartesian strength. This is true if $G$ has the form
$\Sigma_{t}\boxtimes_{p}\Delta_{s}$ since the right adjoints $\boxtimes_{p}$
and $\Delta_{s}$ preserve limits and $\Sigma_{t}$ preserves connected
limits. 
\end{rem}

\begin{rem}
Since each slice is locally subcartesian closed, one has a strength
$\circledast_{\mathcal{E}/B}$ taking any $K\colon K_{0}\to B$ in
$\mathcal{E}/B$ and $a\colon A\to X$ in $\left(\mathcal{E}/B\right)/x\cong\mathcal{E}/X$
for any $x\colon X\to B$ to $K_{0}\otimes_{B}A\to A\to X$ again
an object of $\left(\mathcal{E}/B\right)/x\cong\mathcal{E}/X$.
\end{rem}

We now consider the (replete) image of the pseudofunctor of Prop.
\ref{polytocat} on strong functors and transformations given by the
following data.
\begin{defn}
A \emph{subcartesian polynomial functor} is a functor $P\colon\mathcal{E}/I\to\mathcal{E}/J$
which is bunched strong, and isomorphic to a functor of the form $\Sigma_{t}\boxtimes_{p}\Delta_{s}$
for some triple $\text{\ensuremath{\left(s,p,t\right)}}$ via a cartesian
strong transformation.
\end{defn}

We can now state what the biequivalence of polynomials and polynomial
functors should be in this more general subcartesian setting. We warn
the reader that (perhaps confusingly) cartesian natural transformations
of polynomial functors correspond to 2-cells of polynomial diagrams
whose vertical morphism is invertible, and hence subcartesian 2-cells\footnote{To explain this, we remind the reader that the local subcartesian
structure lies on $\mathcal{E}$ not on $\mathbf{Cat}$, thus we have
no natural notion of ``subcartesian natural transformations''.}.
\begin{thm}
\label{mainbiequiv} For any locally subcartesian closed category
with finite limits $\left(\mathcal{E},\otimes\right)$ the bicategory
of polynomials $\mathbf{Poly}\left(\mathcal{E},\otimes\right)$ is
in biequivalence with the 2-category of subcartesian polynomial functors
$\mathbf{PolyFun}\left(\mathcal{E},\otimes\right)$ and cartesian
strong natural transformations. Moreover, under this equivalence subcartesian
2-cells correspond to cartesian natural transformations.
\end{thm}

\begin{proof}
We have already shown in Prop. \ref{polytocat} that we have a canonical
pseudofunctor $\mathbf{Poly}\left(\mathcal{E},\otimes\right)\to\mathbf{Cat}$.
Note also that the terminal objects of slice categories (identity
morphisms) are preserved by operations $\Delta$ and $\boxtimes$.
The analogue of \cite[Lemma 2.2]{gambinokock} holds for subcartesian
polynomial functors (where dependent products are replaced with dependent
subproducts) by forming the subpullback, and applying the Beck condition
of Cor. \ref{beck3}. It follows that the evaluation fibration of
\cite[Prop. 2.4]{gambinokock} (with subcartesian polynomial functors)
allows factoring transformations of subcartesian polynomials through
the fibration's cartesian lifts, thereby eliminating $\Sigma_{t}$
expressions (this is equivalent to working fibrewise) so that it suffices
to work with vertical transformations. Moreover, both factors remain
bunched strong: the lift is bunched strong by the preceding polynomial
description, and the vertical factor inherits the cartesian and subcartesian
strength conditions through the ordinary pullback defining the evaluation's
cartesian lift.

The subcartesian version of \cite[Lemma 2.6]{gambinokock} is again
an enriched Yoneda embedding, and the fact that it is fully faithful
means (with respect to subcartesian strength)
\[
\mathcal{E}/I\left(s,s'\right)\cong\mathbf{Cat}^{\mathcal{E}}_{\textnormal{str}}\left(\boxtimes_{!}\boxtimes_{s'}\Delta_{s'},\boxtimes_{!}\boxtimes_{s}\Delta_{s}\right).
\]
The more complex part is the analogue of \cite[Prop 2.8]{gambinokock}
which we now explain. Since we have already defined the pseudofunctor,
it remains to check any $\alpha\colon\boxtimes_{p'}\Delta_{s'}\Rightarrow\boxtimes_{p}\Delta_{s}$
cartesian strong (and thus also subcartesian strong) comes from a
subcartesian vertical polynomial 2-cell. Firstly we compose with $\boxtimes_{!_{B}}$
and apply Yoneda to recover $j\colon s\to s'$ with $s=s'j$, so it
remains to check that $p'j=p$. To see this, first note that by adjointness
we have the commuting square at $\pi_{I}\colon B\times I\to I$
\[
\xymatrix@=1em{\left(\mathcal{E}/B\right)\left(1_{B},\boxtimes_{p'}\Delta_{s'}\left(\pi_{I}\right)\right)\ar[d]_{\textnormal{adj}}\ar[rrr]^{\alpha_{\pi_{I}}\circ\left(-\right)} &  &  & \left(\mathcal{E}/B\right)\left(1_{B},\boxtimes_{p}\Delta_{s}\left(\pi_{I}\right)\right)\ar[d]^{\textnormal{adj}}\\
\left(\mathcal{E}/I\right)\left(s',\pi_{I}\right)\ar[rrr]_{\left(-\right)\circ j} &  &  & \left(\mathcal{E}/I\right)\left(s,\pi_{I}\right)
}
\]
Moreover, since $\alpha$ respects cartesian strength we have the
left square below
\[
\xymatrix@=1em{B\times\boxtimes_{p'}\Delta_{s'}\left(1_{I}\right)\ar[rr]^{\textnormal{st}^{p',s'}_{B,1_{I}}}\ar[d]_{B\times\alpha_{1_{I}}} &  & \boxtimes_{p'}\Delta_{s'}\left(B\times1_{I}\right)\ar[d]^{\alpha_{B\times1_{I}}} & \gamma_{2}\ar[rr]^{\textnormal{st}^{p',s'}_{B,1_{I}}\quad\quad}\ar@{=}[d] &  & \boxtimes_{p'}\Delta_{s'}\left(\pi_{I}\right)\ar[d]^{\alpha_{\pi_{I}}}\\
B\times\boxtimes_{p}\Delta_{s}\left(1_{I}\right)\ar[rr]_{\textnormal{st}^{p,s}_{B,1_{I}}} &  & \boxtimes_{p}\Delta_{s}\left(B\times1_{I}\right) & \gamma_{2}\ar[rr]_{\textnormal{st}^{p,s}_{B,1_{I}}\quad\quad} &  & \boxtimes_{p}\Delta_{s}\left(\pi_{I}\right)
}
\]
simplifying to the right above, where $\gamma_{2}\colon B\times B\to B$
is the right projection. By precomposing this square with the diagonal
$\delta_{B}\colon B\to B\times B$ over $B$ and taking adjuncts we
see $\textnormal{st}^{p,s}_{B,1_{I}}\delta_{B}$ corresponds to $\left(p,s\right)\colon s\to\pi_{I}$
which is equal to $\alpha_{\pi_{I}}\textnormal{st}^{p',s'}_{B,1_{I}}\delta_{B}$
corresponding to $\left(p',s'\right)j\colon s\to\pi_{I}$ and so $p=p'j$.
Finally the analogue of \cite[Prop 2.9]{gambinokock} holds by the
same argument as in the usual cartesian setting: under this Yoneda
correspondence between $\alpha$ and $j$ (which here requires bunched
strong transformations) the underlying natural transformation $\alpha$
is cartesian precisely when its representing map $j$ is invertible. 
\end{proof}

\section{Dependent separated products in nominal sets\label{nom}}

In this section we first choose a fixed countably infinite set $\mathbb{A}$,
whose elements we call atoms. This choice is not important, since
any such $\mathbb{A}$ will be isomorphic to $\mathbb{N}$ resulting
in an equivalent category of nominal sets. We will denote by $\Perm{\left(\mathbb{A}\right)}$
the group of finite permutations of these atoms. We will now recall
the basic definitions of nominal sets \cite{nominal}, show the category
of them admits locally subcartesian closed structure, and then consider
a simple polynomial example.
\begin{defn}
A \emph{nominal set} consists of a set $X$ along with a group action
\[
\Perm{\left(\mathbb{A}\right)}\times X\to X\colon\left(\pi,x\right)\mapsto\pi\cdot x
\]
 satisfying the following finite support property: for every $x\in X$,
there exists a finite subset $S\subseteq\mathbb{A}$ such that for
all $\pi\in\Perm\left(\mathbb{A}\right)$ we have
\[
\bigl(\forall a\in S,\ \pi(a)=a\bigr)\Longrightarrow\pi\cdot x=x.
\]
The intersection of all such finite supports $S$ is called the \emph{least
support} of $x$ and is denoted by $\supp\left(x\right)$.
\end{defn}

\begin{defn}
A morphism of nominal sets $X\to Y$ is an equivariant map, meaning
a function $f\colon X\to Y$ such that for all $\pi\in\Perm\left(\mathbb{A}\right)$
and $x\in X$ we have $f\left(\pi\cdot x\right)=\pi\cdot f\left(x\right)$.
\end{defn}

\begin{rem}
Such equivariant maps $f\colon X\to Y$ have the basic property that
$\supp\left(f\left(x\right)\right)\subseteq\supp\left(x\right)$ for
all $x\in X$.
\end{rem}

To exhibit a locally subcartesian closed structure on $\mathbf{Nom}$
we must first define our choice of subpullbacks. One may first try
using the subobject of the pullback given by just taking those pairs
which are separated, however this fails to satisfy the required coherence
axioms. This issue is addressed by taking into account the finite
support structure on the codomain nominal set. In fact, this construction
has already appeared under the name of relative separation \cite{SimpsonInd,SimpsonProbSheaf},
though the affine base change view was not made explicit.
\begin{defn}
The separated pullback $A*_{X}B$ of a pair of maps $f\colon A\to X$
and $g\colon B\to X$ is given by 
\[
\sum_{x\in X}\left\{ \left(a,b\right)\in A_{x}\times B_{x}\colon\supp\left(a\right)\cap\supp\left(b\right)=\supp\left(x\right)\right\} 
\]
where $A_{x}$ and $B_{x}$ are the fibers over $x\in X$. 
\end{defn}

\begin{rem}
It is worth noting that there is another equivalent presentation of
the separated pullback 
\[
\sum_{x\in X}\left\{ \left(a,b\right)\in A_{x}\times B_{x}\colon\left(\supp\left(a\right)\setminus\supp\left(x\right)\right)\cap\left(\supp\left(b\right)\setminus\supp\left(x\right)\right)=\emptyset\right\} 
\]
based on the view that $X$ contains ``old atoms'' and only requiring
that the ``new atoms'' of $A$ and $B$ are separated. The nominal
structure is inherited from that of the product and given by $\supp\left(a,b\right)=\supp\left(a\right)\cup\supp\left(b\right)$
because the action is given component-wise. The reader will also notice
that any injective equivariant map exactly reflects support, and so
gives no new supporting atoms, thus any separated pullback along such
a map collapses to the pullback.
\end{rem}

\begin{rem}
In constructive versions of nominal sets (where least supports need
not exist) the support conditions are better written in freshness
form. In the slice presentation, the relative support condition $\supp\left(a\right)\cap\supp\left(b\right)=\supp\left(x\right)$
is expressed by the condition $\forall n\in\mathbb{A}\;\left(n\#x\Rightarrow n\#a\vee n\#b\right)$.
Under the equivalence with the families presentation $\mathbf{Nom}\left(X\right)\simeq\mathbf{Nom}/X$
\cite[Section 3]{depab} this is written $\forall n\in\mathbb{A}\;\left(n\#x\Rightarrow n\#\left(x,a\right)\vee n\#\left(x,b\right)\right)$
where $\left(x,a\right)$ and $\left(x,b\right)$ are the corresponding
total elements. Thus only atoms not already present in the context
$x$ are required to be separated.
\end{rem}

\begin{rem}
The separated pullback over the terminal nominal set $X=\mathbf{1}$
is the separated product $\left\{ \left(a,b\right)\in A\times B\colon\supp\left(a\right)\cap\supp\left(b\right)=\emptyset\right\} $
since the terminal nominal set has empty support.
\end{rem}

\begin{lem}
The separated pullback defines a coherent class of subpullbacks, and
thus an affine base change $\nabla_{\left(-\right)}\colon\mathbf{Nom}^{\textnormal{op}}\to\mathbf{Cat}$.
\end{lem}

\begin{proof}
Clearly the separated pullback is symmetric. Taking $f$ to be the
identity on $X$ this collapses to 
\[
\sum_{x\in X}\left\{ \left(x,b\right)\in\left\{ x\right\} \times B_{x}\colon\supp\left(x\right)\cap\supp\left(b\right)=\supp\left(x\right)\right\} \cong\sum_{x\in X}B_{x}\cong B
\]
so that identity maps are stable. Now consider 
\[
\xymatrix@=1em{C*_{A}\left(A*_{X}B\right)\ar[d]_{g''}\myar{p'}{rr} &  & A*_{X}B\ar[d]_{g'}\myar{f'}{rr} &  & B\ar[d]^{g}\\
C\ar[rr]_{p} &  & A\ar[rr]_{f} &  & X
}
\]
Any general element of $C*_{A}\left(A*_{X}B\right)$ namely $\left(c,a,b\right)$
has the properties that $pc=a$ and $fa=gb=x$ along with the two
separateness conditions 
\[
\supp\left(a\right)\cap\supp\left(b\right)=\supp\left(x\right),\quad\supp\left(c\right)\cap\supp\left(a,b\right)=\supp\left(a\right)
\]
The equivariance of $f$ and $p$ gives $\supp\left(x\right)\subseteq\supp\left(a\right)\subseteq\supp\left(c\right)$
and from the product nominal structure we have $\supp\left(b\right)\subseteq\supp\left(a,b\right)=\supp\left(a\right)\cup\supp\left(b\right)$.
It then follows that
\[
\begin{aligned}\supp\left(c\right)\cap\supp\left(b\right) & =\left(\supp\left(c\right)\cap\supp\left(a,b\right)\right)\cap\supp\left(b\right)\\
 & =\supp\left(a\right)\cap\supp\left(b\right)\\
 & =\supp\left(x\right)
\end{aligned}
\]
Conversely, $\supp\left(c\right)\cap\supp\left(b\right)=\supp\left(x\right)$
gives $\supp\left(a\right)\cap\supp\left(b\right)\subseteq\supp\left(x\right)$
which is an equality by equivariance of $f$ and $g$ and so 
\[
\begin{aligned}\supp\left(c\right)\cap\supp\left(a,b\right) & =\supp\left(c\right)\cap\left(\supp\left(a\right)\cup\supp\left(b\right)\right)\\
 & =\left(\supp\left(c\right)\cap\supp\left(a\right)\right)\cup\left(\supp\left(c\right)\cap\supp\left(b\right)\right)\\
 & =\supp\left(a\right)\cup\supp\left(x\right)\\
 & =\supp\left(a\right)
\end{aligned}
\]
thus the double subpullback may be identified with triples $\left(c,pc,b\right)$
such that $\supp\left(c\right)\cap\supp\left(b\right)=\supp\left(x\right)$
and the component $pc$ forgotten, which is the desired outer subpullback.
\end{proof}

The operation of separated pullback $\nabla_{p}$ along any map $p\colon E\to B$
has the following right adjoint $\boxtimes_{p}$. As with separated
pullbacks, we must take into account the support of base points. 
\begin{defn}
Let $p\colon E\to B$ and $f\colon A\to E$ be morphisms of nominal
sets. The \emph{dependent separated product} $\boxtimes_{p}\left(f\right)$
of $f$ along $p$ is the set of those pairs
\[
\left(b,s_{b}\colon E_{b}\rightharpoonup_{E}A\right)
\]
where $b\in B$ and $s_{b}$ is a partial section of $f$, equipped
with the group action $\pi\star\left(b,s_{b}\right):=\left(\pi b,\pi\ast s_{b}\right)$
where $\pi\ast s_{b}$ is the partial section of $f$ over $E_{\pi b}$
defined only for $e\in\pi\left(\dom{s_{b}}\right)$ by $\left(\pi\ast s_{b}\right)\left(e\right):=\pi\left(s_{b}\left(\pi^{-1}e\right)\right)$,
such that simultaneously:
\begin{enumerate}
\item $\left(b,s_{b}\right)$ has finite support;
\item the domain of $s_{b}$ is precisely
\[
\left\{ e\in E_{b}\colon\supp\left(e\right)\cap\supp\left(b,s_{b}\right)=\supp\left(b\right)\right\} ;
\]
\item we have
\[
\supp\left(b,s_{b}\right)\setminus\supp\left(b\right)=\bigcup_{e\in\dom{s_{b}}}\left(\supp\left(s_{b}\left(e\right)\right)\setminus\supp\left(e\right)\right);
\]
\end{enumerate}
with the equivariant projection $\left(b,s_{b}\right)\mapsto b$ to
$B$.
\end{defn}

\begin{rem}
The axiom (2) is the main difference from the dependent product: namely
we only allow $e\in E_{b}$ whose new atoms are separated from the
section. This forces us to use partial sections, which then requires
(3) to establish canonicity.
\end{rem}

It is worth noting here that the separated product is well known to
be an instance of Day convolution, and thus gives a closed monoidal
structure \cite{dayconvolution}. This is actually true of separated
pullbacks as well but in a fibred sense, which gives an abstract way
to see that the above construction $\boxtimes_{p}$ is right adjoint
to $\nabla_{p}$. Unfortunately, the theory of fibred convolutions
is far too large to detail here, and so we will give a sketch proof
of this adjunction instead. This proof is largely similar to that
of the adjointness of separated product and wand \cite[Theorem 3.13]{nominal}.
\begin{prop}
The dependent separated product $\boxtimes_{p}$ is right adjoint
to separated pullback $\nabla_{p}$ for every $p\colon E\to B$.
\end{prop}

\begin{proof}
For any $f\colon A\to E$ the first condition ensures elements of
$\boxtimes_{p}\left(f\right)$ have finite support. Moreover, the
displayed nominal action preserves the defining conditions. Since
the projection is equivariant it follows $\boxtimes_{p}\left(f\right)\to B$
is an object of $\mathbf{Nom}/B$. Given $h\colon X\to\boxtimes_{p}\left(f\right)$
over $B$ write $h\left(x\right)=\left(qx,s_{x}\right).$ The uncurried
map
\[
\widetilde{h}\colon E*_{B}X\to A,\qquad\widetilde{h}\left(e,x\right)=s_{x}\left(e\right)
\]
over $E$ is well defined by condition (2). Conversely, given $k\colon E*_{B}X\to A$,
fix $x\in X$ and write $b=qx$. The currying construction gives a
nominal extension $e\mapsto k\left(e,x\right)$ from the separated
fibre $e\#_{b}x$ to the unique finitely supported partial section
$\ensuremath{s_{x}}$ whose domain is determined by condition (2)
and whose relative support is determined by condition (3). Existence
and uniqueness of this extension are obtained from the finite-tuple
form of the freshness theorem \cite[Theorem 3.11]{nominal}. Here
one freshens $x$ over $b$ away from the chosen $e$, and equivariance
of $k$ shows independence of the freshening. Moreover, currying is
well defined since $\supp\left(b,s_{x}\right)\subseteq\supp\left(x\right)$.
These two constructions are inverse since uncurrying after currying
is immediate from the construction, while currying after uncurrying
follows from the uniqueness forced by conditions (2) and (3).
\end{proof}

\begin{example}
There are a few special cases of this construction worth detailing
here. If we restrict to maps $p$ of the form $\textnormal{pr}_{1}\colon G\ast N\to G$
we recover the ``multiplicative dependent product'' of Sch{\"o}pp--Stark
\cite{starknom}. More simply, if we take $p$ to be the unique map
$E\to\mathbf{1}$ and $f=\textnormal{pr}_{1}\colon E\times Z\to E$
for any input $Z$ this reduces to the usual ``wand'' operation
$E\mathbin{-\!\!*}\left(-\right)$ \cite{nom3ax}, which is right
adjoint to the separated product monoidal structure $E\ast\left(-\right)$.

When $E=\mathbb{A}$, this right adjoint is the name abstraction functor
$\left[\mathbb{A}\right]\left(-\right)$\footnote{It is well known that this $\mathbb{A}$-name abstraction functor
has a further right adjoint \cite{Menni}, and one may define various
notions of abstraction for general nominal sets \cite{genabs}.}. More generally, for any nominal set $X$ we can take $p$ to be
the projection $X\ast\mathbb{A}\to X$ in which case $\boxtimes_{p}$
is canonically isomorphic to the usual dependent name abstraction
operation \cite{depab}. 

To make this construction clear, consider taking the same projection
map $p$ and applying $\Pi_{p}$ and $\boxtimes_{p}$ to an equivariant
map $f\colon\Phi\to X\ast\mathbb{A}$ whose fiber over each $\left(x,a\right)$
with $a\#x$ is denoted $\varphi\left(x,a\right)$. A section over
$X$ of $\Pi_{p}\left(f\right)$ is an equivariant choice $x\mapsto t_{x}$
of supported total sections of $f$ over $p^{-1}\left(x\right)$ such
that $t_{x}\left(a\right)\in\varphi\left(x,a\right)$ for all $a\#x$\footnote{In the propositional setting where $f$ is an equivariant inclusion
this reduces to $\forall x\;\forall a\#x\;\varphi\left(x,a\right)\neq\emptyset$.
Moreover, for each $x\in X$, the inner condition is equivalent to
the freshness quantification \reflectbox{$\mathsf{\uppercase{N}}$}$a\;\varphi\left(x,a\right)\neq\emptyset$
of Gabbay and Pitts \cite{GabPitts}.}. In contrast, a section over $X$ of $\boxtimes_{p}\left(f\right)$
is an equivariant choice $x\mapsto s_{x}$ of supported partial sections
of $f$ over $p^{-1}\left(x\right)$ such that $s_{x}\left(a\right)$
is defined and $s_{x}\left(a\right)\in\varphi\left(x,a\right)$ for
all $a\#\left(x,s_{x}\right)$. This version is arguably more natural
in programming semantics as the freshness condition is now relative
to the support of the rule or program $s_{x}$.
\end{example}

\begin{example}
Consider the polynomial diagram in $\mathbf{Nom}$ involving application
and abstraction given by
\[
\mathbf{1}\leftarrow\mathbb{A}+\left\{ \left(\mathsf{app},0\right),\left(\mathsf{app},1\right)\right\} \overset{p}{\longrightarrow}\mathbb{A}+\left\{ \mathsf{app}\right\} +\left\{ \mathsf{abs}\right\} \rightarrow\mathbf{1}
\]
with $p\left(a\right)=\mathsf{abs}$ for all $a\in\mathbb{A}$, and
$p\left(\mathsf{app},0\right)=p\left(\mathsf{app},1\right)=\mathsf{app}$.
Since the category is both locally cartesian closed and locally subcartesian
closed, any polynomial diagram admits two interpretations as a polynomial
functor. In the cartesian sense the functor is given by $P^{\times}\left(X\right)=\mathbb{A}+X\times X+X^{\mathbb{A}}_{\textnormal{fs}}$
where $X^{\mathbb{A}}_{\textnormal{fs}}$ is the nominal set of finitely
supported functions $\mathbb{A}\to X$, whereas in the subcartesian
sense it is $P^{*}\left(X\right)=\mathbb{A}+X\times X+\left[\mathbb{A}\right]X$
where $\left[\mathbb{A}\right]X$ is the name abstraction functor\footnote{The cartesian product $X\times X$ appears in both cases since the
application component is induced by the trivial group action on $\mathbf{1}+\mathbf{1}$.}. This is Pitts\textquoteright{} nominal algebraic functor for the
untyped lambda-calculus signature, with initial algebras coinciding
with nominal sets of abstract syntax trees modulo $\alpha$-equivalence
\cite[Section 8.5]{nominal}. Similar constructions to this have been
given with cartesian contexts, using polynomial functors constructed
from the locally cartesian closed structure of a presheaf category
\cite[Section 4.2]{ArkorBind}.
\end{example}

\begin{example}
Consider the functor $T\colon\mathbf{Nom}\to\mathbf{Nom}$ taking
a nominal set $X$ to $\Sigma_{n\in\mathbb{N}}\left[\mathbb{A}^{\#n}\right]X$
where $\mathbb{A}^{\#n}$ is the nominal set of distinct $n$-tuples
of atoms. This arises from the polynomial
\[
\mathbf{1}\leftarrow\Sigma_{n\in\mathbb{N}}\mathbb{A}^{\#n}\overset{\mathsf{len}}{\longrightarrow}\mathbb{N}\rightarrow\mathbf{1}
\]
with both the unit $\eta_{X}\left(x\right)=\left\langle \right\rangle x$
and multiplication $\text{\ensuremath{\mu_{X}\left(\left\langle \mathbf{a}\right\rangle \left\langle \mathbf{b}\right\rangle x\right)=\left\langle \mathbf{a},\mathbf{b}\right\rangle x}}$
(with representatives chosen such that $\mathbf{a}\#\mathbf{b}$)
defining cartesian strong cartesian natural transformations. This
is an instance of a subcartesian polynomial monad, with algebras amounting
to a nominal set $X$ with an equivariant map $\left[\mathbb{A}\right]X\to X$.
\end{example}

\begin{example}
Consider the linear polynomial $\mathbf{1}\overset{!}{\leftarrow}\mathbb{A}^{\#2}\overset{\textnormal{id}}{\rightarrow}\mathbb{A}^{\#2}\overset{!}{\rightarrow}\mathbf{1}$
in $\mathbf{Nom}$, giving the polynomial functor $P\left(X\right)=\mathbb{A}^{\#2}\times X$
in both the cartesian and subcartesian sense. There is an invertible
natural endotransformation $\alpha$ on this polynomial functor with
equivariant components $\alpha_{X}\colon\mathbb{A}^{\#2}\times X\to\mathbb{A}^{\#2}\times X$
given by $\alpha_{X}\left(a,b,x\right)=\left(b,a,\left(ab\right)x\right)$.
This transformation is cartesian, and respects subcartesian strength
(but not cartesian strength). However, it does not come from a vertical
morphism of polynomials, since the only vertical map is the identity.
This counterexample explains why the assumption of respecting cartesian
strength is necessary to show the middle morphisms are respected and
complete the proof of Theorem \ref{mainbiequiv}.
\end{example}

\begin{rem}
In categories such as nominal sets with both a dependent product and
dependent subproduct structure the inclusion $\nabla_{p}\Rightarrow\Delta_{p}$
corresponds to a mate $\Pi_{p}\Rightarrow\boxtimes_{p}$ from the
dependent product to dependent subproduct.
\end{rem}

\section{Future directions\label{future}}

Due to the centrality of locally cartesian closed categories, this
framework gives rise to a large number of research directions the
reader may pursue. Perhaps the most apparent is generalizing Seely's
proof \cite{Seely} (with appropriate coherence issues \cite{Seelycoh})
to show that a categorical semantics for dependent affine type theory
is indeed given by locally subcartesian closed categories, also understanding
the connections with other proposed semantics \cite{riley}. In light
of our clear semantics, one may consider formal verification showing
that programs are both functionally correct and respect resources,
noting the use of separation logic in heap memory is reminiscent of
the logic of nominal sets.

Another possible direction is studying examples and properties of
subcartesian polynomial functors (such as those which require notions
of freshness where a name or label is used at most once, for instance
distinctly labeling leaves of a tree) and finding recognition theorems
for when a functor is subcartesian polynomial. This includes considering
free polynomial monads and subcartesian W-types, and subcartesian
versions of familial and parametric right adjoint functors.

Finally one may consider these categories as bases for enrichment,
especially for topoi. This is relevant since one requires what McLarty
calls the ``fundamental theorem of topos theory'' \cite{McLarty}
and thus the compatibility with slicing inherent in our structure.
In this direction, one may also investigate the case of $Q$-sets
and sheaves on quantales, especially in the case where the given quantale
$Q$ is locally subcartesian. This further suggests links with lopoi
\cite{tenorio} and a ``fundamental theorem of lopos theory'' (however
the precise link with equalizers involving a global tensor $\otimes_{\mathbf{1}}$
is currently unclear, and may only be reasonable in convolutional
settings).

The author's interest is the observation that the adjunction $\nabla_{p}\dashv\boxtimes_{p}$
we constructed in nominal sets is better seen as an artifact of  the
locally subcartesian structure on $\mathbf{FI}^{\textnormal{op}}$
from Example \ref{FI} lifting to a closed one on $\left[\mathbf{FI},\mathbf{Set}\right]$.
In fact, under the standard equivalence taking a nominal set $X$
to a sheaf $\mathbf{FI}\to\mathbf{Set},$ $S\mapsto\left\{ x\in X:\supp{\left(x\right)}\subseteq S\right\} $
and with a cospan of nominal sets instead seen as sheaves $F$ and
$G$ over $K$, the separated pullback may be seen as a convolution
of the form
\[
\int^{\underset{\supp\left(t\right)=T}{\left(T,t\right)\in\textnormal{el}\left(K\right)}}\int^{\alpha:A\to T,\beta:B\to T\in\mathbf{FI}^{\textnormal{op}}}\mathbf{FI}^{\textnormal{op}}\left(-,\alpha\otimes_{T}\beta\right)\times F\left(A\right)_{K\left(\alpha\right)\left(t\right)}\times G\left(B\right)_{K\left(\beta\right)\left(t\right)}.
\]
We will cover the theory of such fibred convolutions in a future paper
in the more general setting of indexed monoidal categories. The above
example is in fact relatively simple due to the base category being
locally subcartesian and pullback preserving functors $\mathbf{FI}\to\mathbf{Set}$
being familial in the two dimensional sense \cite[Remark 2.3]{WalkerFam},
namely a groupoid-indexed colimit of representables. 

\bibliographystyle{abbrv}
\bibliography{references}

\end{document}